\documentclass{elsarticle}
\usepackage{eurosym}
\usepackage{amsthm}
\usepackage{graphicx}
\usepackage{amsmath,amssymb,amsfonts}
\usepackage{algorithm}
\usepackage{hyperref}
\usepackage{cleveref}

\newcommand{\betanit}{p_{\textrm{init}}}

\newcommand{\revision}{}%

\renewcommand{\geq}{\geqslant}

\renewcommand{\leq}{\leqslant}

\DeclareMathOperator*{\argmax}{arg\,max}

\newtheorem{assumption}{Assumption}
\newtheorem{theorem}{Theorem}
\newtheorem{proposition}[theorem]{Proposition}
\newtheorem{corollary}[theorem]{Corollary}
\newtheorem{lemma}[theorem]{Lemma}
\theoremstyle{definition}
\newtheorem{definition}[theorem]{Definition}
\theoremstyle{remark}
\newtheorem{remark}[theorem]{Remark}

\def\mybigtimes{\mathop{\mathchoice{%
   \vcenter{\hbox to10bp{\vrule height15bp width0pt \pdfliteral{
   q 1 J .8 w 0 1 m 10 14 l S 0 14 m 10 1 l S Q
}\hss}}}{%
   \vcenter{\hbox to10bp{\kern1bp\vrule height10bp width0pt \pdfliteral{
   q 1 J .65 w 0 0 m 8 10 l S 0 10 m 8 0 l S Q
}\hss}}}{\times}{\times}%
}}

\begin{document}

\title{Optimal Strategy against Straightforward Bidding in Clock Auctions}
\author[1,2]{Jad Zeroual}
\ead{jad.zeroual@polytechnique.edu}
\author[2]{Marianne Akian}
\ead{marianne.akian@inria.fr}
\author[1]{Aurélien Bechler}
\ead{aurelien.bechler@orange.com}
\author[1]{Matthieu Chardy}
\ead{matthieu.chardy@orange.com}
\author[2]{St\'ephane Gaubert}
\ead{stephane.gaubert@inria.fr}

\affiliation[1]{organization={Orange Labs},
city = {Chatillon},
country = {France}}

\affiliation[2]{organization={INRIA and CMAP, \'Ecole polytechnique},
city = {Palaiseau},
country = {France}}

\begin{abstract}
We study a model of auction representative of the 5G auction in France. We determine the optimal strategy of a bidder, assuming that the valuations of competitors are unknown to this bidder and that competitors adopt the straightforward bidding strategy. Our model is based on a Partially Observable Markov Decision Process (POMDP). This POMDP admits a concise statistics, avoiding the solution of a dynamic programming equation in the space of beliefs. In addition, under this optimal strategy, the expected gain of the bidder does not decrease if competitors deviate from straightforward bidding. We illustrate our results by numerical experiments, comparing the value of the bidder with the value of a perfectly informed one.
\end{abstract}
\maketitle

\section{Introduction}
\subsection{Context}
The acquisition of frequency spectrum is a vital aspect for telecommunications companies, as their core operations and success rely on these resources. Spectrum auctions have emerged as a prominent method for allocating these valuable bandwidths.
They have undergone significant evolution since their introduction, with various auction models being implemented over time ~\cite{auction_theory}. Initially, sealed-bid auctions were the preferred method for allocating spectrum rights. In this model, bidders would submit their bids simultaneously without knowing the bids of their competitors, the highest bidder winning the auction. However, this model was found to have limitations. For example, in auctions with several frequency bandwidths at stake (which is usually the case), bidders ended up paying very different amounts for the same goods (see page 8 of ~\cite{auction_to_work}). To address these limitations, auction models have evolved to accommodate more complex scenarios. One notable development was the introduction of combinatorial auctions, which allow bidders to bid on packages of items rather than individual items ~\cite{ausubel_ascending_nodate}. This innovation significantly improved the efficiency of spectrum allocation by enabling bidders to express their preferences for specific combinations of spectrum licenses. Among the various combinatorial auction formats, the \textit{combinatorial clock auction} (CCA) has emerged as a popular choice for spectrum auctions during the 4G era (see~\cite[Chapter 9]{nera}). The CCA combines the advantages of the clock auction, where prices increase in rounds until demand equals supply, with the flexibility of combinatorial bidding. This format allows bidders to adjust their bids in response to changing prices, while also considering complementarities and substitutabilities among the frequency bandwidths. More precisely, the CCA presents two phases: the first is a clock stage featuring a multi-round auction where each bidder bids for a unique package at each round and the price of each individual item is raised when the demand exceeds the offer of said item. When the bidding reaches a round with no excess demand, the CCA enters its supplementary round where bidders bid a last time for a package of items with respect to the prices of the precedent phase ~\cite{ausubel_clock_auction}. In this context, ~\cite{auction_theory,milgrom_substitute_2009} provide an introduction to auction theory and ~\cite{levin_properties_2016} further explores this issue, suggesting that a truthful strategy for CCA would not be optimal in the general setting but yield efficient results. 
In the 5G era, new takes on this format have emerged and gained popularity mainly due to its complexity. This was emphasized by NERA, the organizers of the 2013 4G Singapore auction, when they believed CCA would give too much of an advantage to well-established competitor whilst the regulator was hoping to strengthen the impact of smaller players (see~\cite[Chapter 9]{nera}). therefore, they implemented a new clock auction format: the Clock Plus auction. In those cases, bidders can only bid on the number of items they want to acquire within a category rather than the exact bundle of items. This type of format was also used as part of the 5G Auctions in France in 2020 with only one category of frequency bandwidth ~\cite{attribution_freq}. Indeed, telecommunication operators took part in a Clock Plus auction to decide how many frequency bandwidths will be allocated to each before participating in an assignment auction.

The evolution of spectrum auctions has also been motivated by the matter of \textit{optimality} which can be defined differently: it could be to maximize revenue for the auctioneer ~\cite{optimal_auctions}, to maximize the \textit{fairness} of the auction ~\cite{jain1998quantitative}
or to maximize one player's profit selfishly. One way it has been studied is through prophet inequalities, which are inequalities between a strategic allocation and the optimal allocation deduced by an \textit{oracle}, a player who is perfectly informed and has perfect knowledge on future states. Those inequalities are widely studied in the literature for different forms of bidder's preference but almost always in a mechanism design perspective, so as to maximize social welfare ~\cite{dutting_olog_2020,eden_constant_2023}. %
However, few have studied the question of optimality in a competitive auction where a player wants to selfishly maximize its own utility. An example of such a study can be found in ~\cite{noauthor_mixed_nodate} which proposes a Mixed Linear Integer Programming approach in a perfect information setting. Nonetheless, the perfect information setting case presents a difficulty: companies do not disclose their valuations, i.e.\ their preferences among frequency bandwidths, to their competitors in order to keep a competitive edge. As a matter of fact, such valuations can provide strategic information on a company's long-term projects. Thus, we can only have some coarse estimates of the opponent valuations, noting that data are generally insufficient to infer such estimations ~\cite{attribution_freq}. Hence, more practical approaches in imperfect information setting are used such as Partially Observable Markov Decision Processes (POMDP). In ~\cite{boutilier_pomdp}, this framework is used for inferring the utility of each player. Those models have been rare for modeling auction competition. 

One of the possible reasons is that such models require to make assumptions on the behavior of competitors. A commonly studied behavior is the  \textit{Straightforward Bidding} (SB) introduced in
~\cite{auction_to_work}. %
It is characterized by its focus on immediate gains without considering future implications. As a matter of fact, this strategy has been used as a baseline and is the default strategy to model auction in the literature. For instance, ad auctions are modeled among SB-players in ~\cite{cai_mdp} and the SB strategy is used to compare the agents performance in ~\cite{reeves_exploring_2005}. This strategy has been at the heart of studies for two reasons. First, it is an intuitive behavior to adopt in auctions, in ~\cite{ausubel_ascending_2002}, the authors mention that it is consistent with experimental auctions. And second, despite its simplicity, it can be optimal in various situations. For instance, SB is proven to be a weakly dominant strategy in auctions of one item ~\cite{bikhchandani_competitive_1997}. Moreover, if each bidder demands a single item and has no preference for any of them, SB is a Bayes-Nash equilibrium ~\cite{Wellman}. What is more, this strategy yields to a situation comparable to a competitive equilibrium when the goods are substitutes for all bidders ~\cite{auction_to_work}, which is the case for the 5G auction we study ~\cite{attribution_freq}. This strategy also ensures truthfulness. A bidder bids truthfully when they are willing to share their valuation with their competitors, SB is thus an example of truthful strategies ~\cite{auction_to_work}.

This paper aims to explore what could be the optimal response against SB opponents for a bidder taking part in a clock auction such as the French 5G auction. Indeed, the SB strategy is interesting as it is a simple yet efficient strategy ~\cite{bulow_winning_nodate} for the auction at hand.

\subsection{Contribution}
We introduce a POMDP formulation for a particular kind of clock auctions, with {\em identical items}. We compute the optimal strategy of a player against straightforward bidders which we call the Bellman strategy. Our first result (\Cref{main_theorem}) provides a simplified expression of the latter under some assumptions on the distribution of the opponents' preferences. It shows that the optimal solution of the POMDP can be expressed through a concise form avoiding the recourse to the belief space. Secondly, we show that the Bellman strategy yields a better result when all players play according to that strategy.
Lastly, we explore the results of this strategy when the assumptions of the theorem are satisfied and show empirical evidence of its performance.

The paper is organized as follows: \Cref{5G} presents the studied auction, the SB strategy and how it is modeled in the rest of the paper. We detail the POMDP model in \Cref{pomdp_model}. It is then solved using dynamic programming equation and we provide a simple form of the solution in \Cref{main_theorem} in the case where the opponents' demands satisfy Markov property (\Cref{MP}). We derive further properties of the optimal strategy in the rest of \Cref{Bellman} under the same assumption.
Our second main result is presented in \Cref{robust}: \Cref{theorem3} states that under the optimal strategy, the expected gain of the bidder does not decrease when the other bidders deviate from Straightforward Bidding, and in particular when all players are utilizing the optimal strategy. This results is valid under \Cref{SM}.
In \Cref{Models}, we give some examples of probabilistic models which satisfy both assumptions of our previous results.
Lastly, \Cref{Example} applies those results to simulated auctions allowing one to compare the performance of this bidding strategy to the one of a perfectly informed bidder.

\section{5G auction in France}\label{5G}
\subsection{Auction mechanism}
We model an auction amongst $n$ players for $m$ items inspired by the clock auction held for the 5G auction in France in 2020 ~\cite{attribution_freq}. We denote each player by an integer $i \in \{1, \dots, n\}$.
The auction begins at a certain price $\betanit \geq 0$. A price increment $\Delta P > 0$ between successive rounds is fixed in advance.

The auction mechanism is the following:
\begin{enumerate}
    \item The auction starts at price $\betanit$, we set $p \longleftarrow \betanit$
    \item Each player $i$ asks for a number of items $\delta_i(p)$. This is their \textit{demand} or \textit{bid}.
    All demands are simultaneous.
    \item 
    We check if the total demand does not exceed the number of items, i.e.\ if $\sum_{i = 1}^n \delta_i(p) \leq m$:
    \begin{itemize}[-]
        \item If it is the case, the auction terminates and each player $i$ receives $\delta_i(p)$ items and pays $\delta_i(p) \times p$.
        \item otherwise, the price is raised, $p \longleftarrow p+\Delta P$,
          and the auction moves to the next round (resuming from step 2).
    \end{itemize}
\end{enumerate}

\begin{remark}
    In practice, $\Delta P$ can vary during the auction. This would allow the auctioneer to account for consequent drop in demand from the players. In our case, we suppose the increment is constant.
\end{remark}

Moreover, the auction presents an \emph{eligibility} rule: it is mandatory for player $i$'s demand to be non-increasing, i.e.\ $\forall p \geq 0, \delta_i(p+\Delta P) \leq \delta_i(p)$.

This auction presents both \emph{public information} and \emph{private information}:
\begin{itemize}[-]
    \item In our setting, the total demand is revealed at the end of the round. At round $t$, the past demands are public information, i.e.\ $\{ \sum_{i = 1}^n \delta_i(\betanit + s\Delta P) | 0 \leq s < t\}$  is known by all players.
    \item Each player $i$ has a budget $B_i$. This is a private information. Every player is under a \emph{budget constraint}: one's payment cannot exceed one's private budget.
\end{itemize}

Following the spirit of the literature ~\cite{auction_theory,auction_to_work,milgrom_substitute_2009,correa}, we model the preferences of each player $i$ by a \emph{valuation} function $v_i : \{0, \dots, m\} \longrightarrow \mathbb{R}_+$. This valuation represents a maximal price that the player $i$ is willing to pay for $k$ items. It is also private information.

\begin{assumption}
  We suppose the valuations are normalized, i.e.\
  $v_i(0)=0$ for all $i \in \{1, \dots, n\}$.
\end{assumption}

In other words, the maximal price any player is willing to pay for acquiring nothing is \EUR{0}.

Furthermore, we introduce the utility $u_i(k, p) = v_i(k) - kp$. Each agent $i$ wants to maximize this utility within the constraints of the auction.

\subsection{Straightforward bidding}
In our study, we suppose all but one player play according to a strategy called Straightforward Bidding (SB) ~\cite{auction_to_work}. %
This strategy is a myopic strategy: it consists in maximizing one's utility at each round, as if the auction would terminate immediately. The player handles possible tie breaks by taking the lowest number of items that maximizes their utility.
In our case in which there is a single type of items, SB can be formulated as follows.
\begin{definition}\label{SB_def}
  The player $i$ is said to be playing SB if
    $$\forall p \geq 0, \delta_i(p) = \min \Big(\argmax_{0 \leq k \leq m}\{ v_i(k) - kp \}\Big)$$
\end{definition}
One can notice that in \Cref{SB_def},
$\delta_i$ only depends on the map 
$p\geq 0\mapsto \max_{0 \leq k \leq m}\{ v_i(k) - kp \}$. This is precisely the Legendre-Fenchel transform of $v_i$ (up to a change of sign), restricted to the non-negative real numbers. Hence, $\delta_i$ only depends on the non-decreasing concave hull of the private valuation $v_i$.
This is formalized by the following result.

\begin{proposition}\label{Prop_concave} Suppose player $i$ plays SB. 
  Let
  \[
  \Breve{v}_i = \inf \{f : \{0, \dots, m\} \rightarrow \mathbb{R} \mid  f \text{ is non decreasing, concave and } f \geq v_i \}
  \]
  and for all $p \geq 0$, define
  \( \Breve{\delta}_i(p) = \min \Big(\argmax_{0 \leq k \leq m}\{ \Breve{v}_i(k) - kp \}\Big)\).
  Then,
  \begin{enumerate}
    
\item Let $k_0 \in \{1, \dots, m-1\}$ such as There exists $p$ satisfying $\delta_i(p) = k_0$. 
    Then, $v_i$ is \emph{locally strictly concave} in $k_0$ (meaning $v_i(k_0) - v_i(k_0-1) > v_i(k_0+1) - v_i(k_0)$). 
    Moreover, for such a $k_0$, $\Breve{v}_i(k_0)=v_i(k_0)$.

\item    for all $p \geq 0$, we have $\delta_i(p) = \Breve{\delta}_i(p)$.
    \end{enumerate}
\end{proposition}

\begin{proof}
    \begin{enumerate}
        \item Let $k_0 = \min(\argmax_{0 \leq k \leq m}\{v_i(k) - k p\})$. Then, 
        \begin{align*}
            v_i(k_0) - k_0 p &\geq v_i(k_0 +1) - (k_0 + 1) p\\
            p &\geq v_i(k_0+1) - v_i(k_0)\\
            v_i(k_0) - k_0 p &> v_i(k_0 -1) - (k_0 - 1) p\\
            v_i(k_0) - v_i(k_0 - 1) &> p
        \end{align*}
        Thus, $v_i(k_0) - v_i(k_0 - 1) > v_i(k_0+1) - v_i(k_0)$. By definition of $\Breve{v}_i$ and since $v_i$ is concave in $k_0$, $v_i(k_0) = \Breve{v}_i(k_0)$.
        \item $\forall k \in \{0, \dots, m\}, \big(\exists p \geq 0, \delta_i(p) = k\big) \implies \big(\Breve{v}_i(k) = v_i(k)\big)$ hence $\Breve{\delta}_i = \delta_i$.
    \end{enumerate}
\end{proof}

\begin{corollary}\label{model_val}%
 Suppose player $i$ plays according to SB. 
We can model their valuation by a function of the form ${v_i}(k) = \sum_{j = 1}^k z^i_j$ where $z^i_1 \geq z^i_2 \geq \dots \geq z^i_m \geq 0$ and in this case, $\delta_i(p) = \sum_{j = 1}^m \mathbf{1}(z^i_j - p > 0)$.
\end{corollary}
\begin{proof} \revision{According to \Cref{Prop_concave},} we can model the valuation \revision{$v_i$} of player $i$ as a non-decreasing and \revision{concave} function $\hat{v}_i$. \revision{Indeed, $\forall p\geq 0, \hat{\delta}_i(p) = \delta_i(p)$, hence considering $\hat{v}_i$ rather than $v_i$ does not change the demands of player $i$. In the rest of the proof, we consider that $v_i = \hat{v}_i$.}

Then, for all $k \in \{0, \dots, m \}$, ${v}_i (k) = \sum_{j = 1}^k z^i_j$ where for all $j \in \{1, \dots, m\}$, $z^i_j = {v}_i(j) - {v}_i(j-1)$. The fact that $v_i$ is non-decreasing, non-negative \revision{and concave} ensures that $z_1^i \geq z_2^i \geq \dots \geq z_m^i \geq 0$.
Furthermore,
\begin{align*}
     \delta_i(p) &= \min \Big(\argmax_{0 \leq k \leq m}\{ {v}_i(k) - kp \}\Big) %
     = \min \Big(\argmax_{0 \leq k \leq m}\{ \sum_{j = 1}^k (z_j^i - p) \}\Big)\\
     &= \begin{cases}
         \max\{j \in \{1, \dots, m\} : z_j^i - p > 0\} &\text{ if the set is not empty}\\
         0 &\text{ otherwise}
     \end{cases} \\
     &= \sum_{j = 1}^m \mathbf{1}(z_j^i - p > 0)%
 \end{align*}
\end{proof}
This corollary can be interpreted geometrically as shown in \Cref{fig:valuation}: the $(z_j^i)$ represent the slopes of the secret valuation $v_i$. An SB demand consists in counting the number of slopes that are steeper than the current price.
\begin{figure}[H]
    \centering
\includegraphics[width=0.5\linewidth]{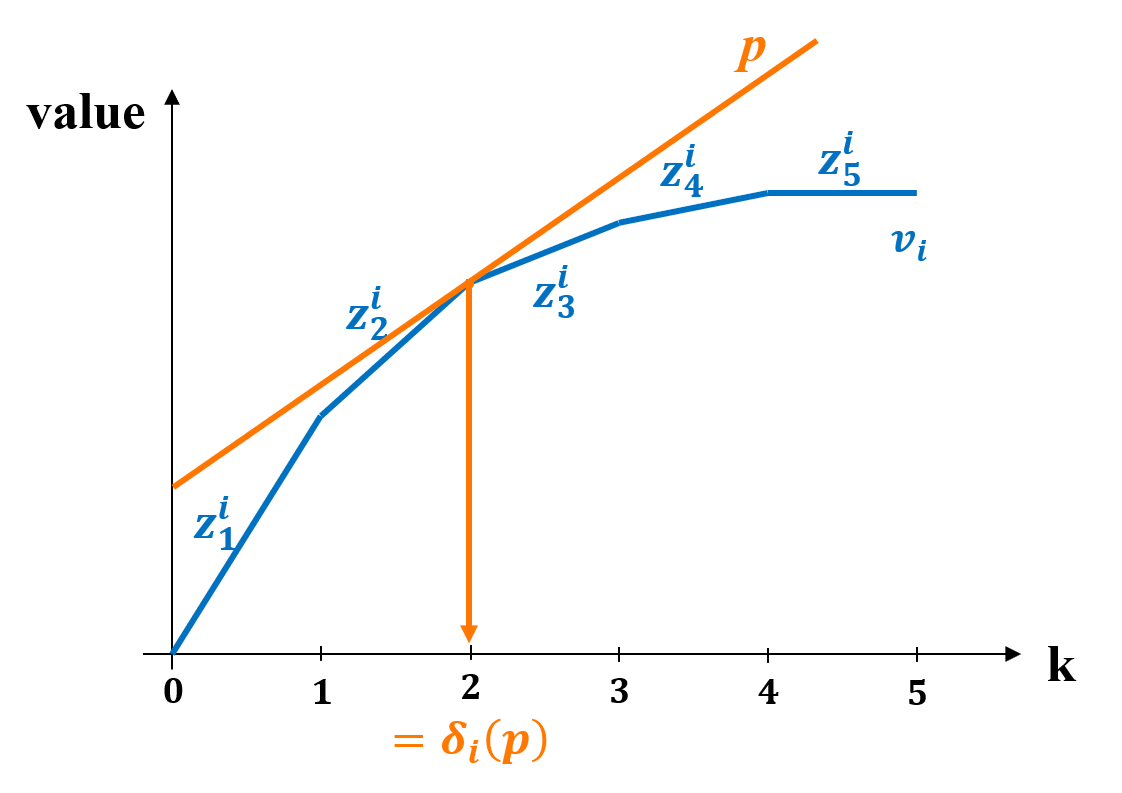}
    \caption{Visualizing the response to a price offer $p$ on the Newton polygon of the valuation map}
    \label{fig:valuation}
\end{figure}
\begin{remark}
    With \Cref{model_val}, we represent the concave hull of the discrete map $k\mapsto v_i(k)$, which is known as the {\em Newton polygon} of $v_i$. The values $(z^i_j)_{1\leq j \leq 5}$ coincide with the successive slopes of the segments of the Newton polygon. They are special (one dimensional) instances of the notion of {\em indifference locus}, studied in~\cite{baldwin_klemperer_tropical} using tools of tropical geometry: those are the prices at which player $i$ is indifferent between acquiring $j$ or $j+1$ items.
The demand at price $p$ can also be interpreted on the Newton polygon. It is the number $k$ (there in \Cref{fig:valuation} $k = 2$) for which there is a supporting line of slope $p$ touching the Newton polygon at point $(k,v_i(k))$.
\end{remark}

Thereafter, \revision{in virtue on \Cref{model_val}, }we will model the valuation of an SB agent as a concave and non-decreasing function as it \revision{ does not reduce the generality of the model}. This \revision{is formalized in} the following assumption:

\begin{assumption}
    Every valuation function $v_i$ is concave and non-decreasing, i.e.\ can be written in the form $v_i(k) = \sum_{j =1}^k z^i_j$ with $z^i_1 \geq \dots \geq z^i_m$.
\end{assumption}

Since in our setting, the total demand is revealed, we model all the SB-players as one super player.

\begin{definition}
Let $J_1$,$\dots$,$J_r$ be players with respective demand functions $\delta_1$, $\dots$,$\delta_r$. We call an aggregation of players $J_1$,$\dots$, $J_r$ a player whose demand function is $\sum_{i = 1}^r \delta_i$.
\end{definition}

\begin{proposition}
    An aggregation of SB-players can be viewed as an SB-player. We call it a super SB-player.
\end{proposition}

\begin{proof}
    Let $r \in \mathbb{N}$. Let $\delta_1$,$\dots$,$\delta_r$ be $r$ demand functions of SB-players. We denote by $v_1$,$\dots$,$v_r$ their valuation functions.

    According to Corollary \ref{model_val}, There exists $(z^{(i)}_j)_{\substack{1 \leq i \leq r \\ 1 \leq j \leq m}}$ such as $\forall i \in \{1, \dots, r\}$, $z^{(i)}_1 \geq \dots \geq z^{(i)}_m \geq 0$, $v_i(k) = \sum_{j=1}^k z^{(i)}_j$ and $\delta_i(p) = \sum_{j = 1}^m \mathbf{1}(z^{(i)}_j - p > 0)$.

    We reorder the sequence $(z^{(1)}_1, \dots, z^{(1)}_m, z^{(2)}_1, \dots, z^{(i)}_j, \dots, z^{(r)}_m)$ to a sequence $(z_1, \dots, z_{rm})$ such as $z_1 \geq \dots \geq z_{rm}$.

    The aggregated demand function of the SB-players is $\delta = \sum_{i = 1}^r \delta_i$. Hence,
    \begin{align*}
        \delta(p) &= \sum_{i = 1}^r \delta_i(p)\\
        &= \sum_{i = 1}^r \sum_{j = 1}^m \mathbf{1}(z^{(i)}_j - p > 0)\\
        &= \sum_{j = 1}^{rm} \mathbf{1}(z_j - p > 0)
    \end{align*}
    
    Let $\forall k \in \{0, \dots, rm\}, v(k) = \sum_{j = 1}^k z_j$. Then, $\delta(p) = \sum_{j = 1}^{rm} \mathbf{1}(z_j - p > 0) = \min\Big(\argmax_{0 \leq k \leq rm}\{ v(k) - kp\}\Big)$. Hence, the aggregation of SB-players can be seen as a SB-player of valuation $v$.
\end{proof}

\begin{remark}
    In this proof, we introduced the valuation $v$ of a super SB-player. One may note that \[ v(k) = \sup_{l_1 + \dots+ l_r = k}\sum_{i = 1}^r v_i(l_i). \]
    This form shows that the valuation of a super SB-player is the sup-convolution of valuations of SB-players. Geometrically, this can be seen on the graph of the super player's valuation, which is the Minkowski sum of the graphs of all the players' valuations.
\end{remark}

\subsection{Playing against a super SB-player}

Our goal is to find a strategy for the non-SB-player in the auction. In the rest, we call the super SB-player \emph{the opponent} and the non-SB-player \emph{the player}. The latter's private valuation is $v$.

\begin{proposition}
    From the player perspective, the auction terminates in a finite number of rounds $R$:
    \[ R \leq 1+ \lceil \Delta P^{-1} \Big( v(1) - \betanit \Big) \rceil =: 1+\overline{R} \]
    The auction terminates in the sense that from round $\overline{R}$, the player can no longer make relevant strategic decisions.
\end{proposition}
\begin{proof}
    From the player perspective, the auction terminates:
    \begin{enumerate}
        \item When the termination condition is satisfied;
        \item When their demand hits $0$.
    \end{enumerate}
    Indeed, after the demand hits $0$, the player cannot make any strategical decision: the eligibility rule forces them to stick to a null demand. Therefore, a player views the auction's horizon as finite.
    
    Let $p_{\overline{R}} = \betanit + \overline{R}\Delta P$ be the price at round $\overline{R}+1$.
    \begin{align*}
        \forall k \in \{0, \dots, m\}, v(k) - kp_{\overline{R}} \leq 0 &\iff \begin{cases}
            k = 0\\
            \text{or}\\
             k \neq 0 \text{ and }\frac{v(k)}{k} \leq p_{\overline{R}}
        \end{cases} 
    \end{align*}
    
    Notice that $k \mapsto \frac{v(k)}{k} = \frac{v(k) - v(0)}{k - 0}$ is non-increasing since $v$ is concave. Thus, $\max_{1 \leq k \leq m} \frac{v(k)}{k} = v(1)$. $p_{\overline{R}} = \betanit + \Delta P (\lceil \Delta P^{-1} \Big( v(1) - \betanit \Big) \rceil) \geq v(1)$. Hence, $\forall k \in \{1, \dots, m\}, \frac{v(k)}{k} \leq p_{\overline{R}}$. Therefore, the demand of the player is necessarily $0$ at round $\overline{R}$.
    \end{proof}

We define $\mathcal{P}= \{ \betanit + k \Delta P | k \in \{0, \dots, \overline{R}\} \}$, the prices that can hold relevance during the auction.

We denote by $\delta$ the opponent's demand function.

In a perfect information setting i.e.\ when the opponent's valuation is public, the optimal policy comes naturally:

\revision{
\begin{proposition}\label{perfect}
    Against an SB-player, for a perfectly informed player, bidding $\argmax_{k \in \{0, \dots, m\}} v(k) - k p_k$ where $p_k = \inf\{ p \in \mathcal{P} : \delta(p) + k \leq m\}$ at each round is optimal.
\end{proposition}
\begin{proof}
    Let $s$ be a strategy for a perfectly informed player. We denote by $k$ their final bid at the terminal price $p \in \mathcal{P}$. Therefore, the player earns a utility of $u(k,p)$ and $k + \delta(p) \leq m$. 
    Let $\kappa = \argmax_{l \in \{0, \dots, m\}} v(l) - l p_l$ where $p_l = \inf\{ p \in \mathcal{P} : \delta(p) + l \leq m\}$. Then, a perfectly informed player bidding $\kappa$ at each round would end the auction at price $p_\kappa$ and $u(\kappa, p_\kappa) \geq 0$ because $v(0) - 0p_0 = 0$.
    Therefore, if $k = 0$, we have $u(k,p) = 0 \leq u(\kappa, p_\kappa)$. Otherwise, $\rho \mapsto u(k,\rho)$ is decreasing, thus $u(k,p) \leq u(k, p_k) \leq u(\kappa, p_\kappa)$. 
    Hence, for a perfectly informed player, playing any strategy $s$ always yields a final utility which is less or equal than the one obtained by bidding $\argmax_{k \in \{0, \dots, m\}} v(k) - k p_k$ at each round.
\end{proof}
}

An optimal strategy for such {an oracle player} is to bid the $k$ that maximizes $v(k) - kp_k$.
Since this result is immediate, the literature primarily examines scenarios where the opponent's valuation is either unknown or revealed through signaling during the auction ~\cite{dutting_olog_2020,eden_constant_2023}. \revision{Furthermore, this is a more sound framework as competitors usually ignore their opponents valuations.}

The following section formally introduces the optimization problem regarding the player's strategy.%
\section{POMDP model}\label{pomdp_model}
We model the opponent's valuation as a random variable $V(k) = \sum_{j = 1}^k Z_j$ with $(Z_j)$ random non-negative variables of known distribution. The $(Z_j)$ must verify $Z_1 \geq \dots \geq Z_{(n-1)m}$ almost surely (a.s.). We denote their demand function $\delta(p) = \sum_{j = 1}^{(n-1)m} \mathbf{1}\{Z_j > p\}$ which is viewed as a random process.

\begin{definition} We model the situation as a Partially Observable Markov Decision Process (POMDP).
\begin{itemize}
    \item[$\bullet$] We denote by $\mathcal{S}$ the state space.
A state $s_t = (t, p_t, k_t, \omega_t) \in \mathcal{S}$ at time $t$ is defined by:
    \begin{itemize}
        \item[-] $t \in \{0, \dots, \overline{R}\}$ is a discrete time, it can be interpreted as the round of the auction.
        \item[-] $p_t\in {\mathcal P}$ is the price at round $t$. The price dynamics is given by $p_0 = \betanit$ and
        $\forall t\in\{0, \dots, \overline{R}\}, p_{t+1} = p_t + \Delta P$.
        \item[-] $k_t\in \{0,\ldots, m\}$ is the player's bid at round $t-1$.
        \item[-] $\omega_t$ is the choice of nature for the opponent's valuation during the auction. We suppose that $\forall t\in\{0, \dots, \overline{R}\}, \omega_{t+1} = \omega_t$ meaning that the opponent's valuation does not change during the auction. This opponent's bid can be computed from $\omega_t$ through a deterministic function $\overline{\delta} : (\omega_t, p_t) \mapsto \delta_t$.
    \end{itemize}
\item[$\bullet$] We denote by $\mathcal{O}$ the set of observations.
The observation at time $t$ is given by %
$o_{t} =(t, p_t, k_t, \delta_{t-1}), \forall t \in \{1, \dots, \overline{R}\}$ where $\delta_t = \bar{\delta}(\omega_t, p_t)$.
We set $k_0 := \argmax_{k \in \{0, \dots, m\}}(v(k) - k \betanit)$ and $o_0 = (0, \betanit, k_0)$.

\item[$\bullet$] We introduce the information vector $\mathcal{I}_t$ which is the information available to the player at $t$: $\mathcal{I}_0 = o_0$ and $\mathcal{I}_t = (o_0, k_0, \dots, o_{t}, k_t)$ for $t \in \{1, \dots, \overline{R}\}$.

\item[$\bullet$] From this information, at round $t$, the player takes an action $u_t = \sigma_t(\mathcal{I}_t)\in  \{0,\ldots, m\}$ with $\sigma_t$ a measurable function and $u_t\leq k_t$. The sequence $(\sigma_t)_{t \geq 0}$ is called an admissible strategy. The state following the action $u_t$ satisfies $ k_{t+1} = u_t$. We thus have $\mathcal{I}_t = (o_0, u_0 \dots, o_{t-1}, u_{t-1}, o_t)$.

\item[$\bullet$] The action $u$ causes the state to change from $s$ to $s'$ with probability $T(s' | s, u)$. As a matter of fact, in this model, all transitions are deterministic:
\[
    T(%
(t', p', k', \omega') \mid%
(t, p, k, \omega), u) = \begin{cases}
        1 &\text{if } \omega' = \omega, k' = u, t'=t+1,  \\
        &\quad p'=p+\Delta P\\
        0 &\text{otherwise}
    \end{cases}
\]
\end{itemize}
\end{definition}

\begin{remark}
    Although the transitions are deterministic, the trajectory $(s_t)_{0 \leq t \leq \overline{R}}$ is not. Indeed, the initial state $s_0$ is random because $\omega_0$, the unknown nature's choice, is viewed as a random variable.
\end{remark}

\begin{remark}
    In the framework we consider, the aggregated demand is made public at each round. Therefore, the aggregated demand $\delta_{t-1}$ is part of the observation at time $t$.
    Other auctions may consider other information as public. 
    For instance, the auctioneer may choose to reveal all individual demands (meaning $\{ \delta_i(\betanit + s\Delta P) | 0 \leq s < t, \; i=1,\ldots, n\}$ is available at round $t$) or only if the auction is continued.
    Nonetheless, we could apply the same approach to auction with different kind of private information by modifying the observation set.
\end{remark}

\noindent \textbf{Optimization problem} The problem is to maximize the player's expected value. Our goal is thus to find an optimal admissible strategy %
i.e.\ optimize with regards to $(\sigma_t)$: %
\begin{equation}\label{objective}
\begin{split}
    \text{maximize }& \mathbb{E}[v(u_\tau) - u_\tau p_\tau] \\
    \text{where }&\tau = \inf\{ t \in \{0,\dots,\overline{R}\} \text{ } | \text{ } \delta_t + u_t \leq m \} \textit{ (termination condition)}
\end{split}
\end{equation}
Although the problem presents a stopping time, the horizon is finite, bounded by $\overline{R}$.

In order to solve such a POMDP, as explained in ~\cite[pp.~185--225]{bertsekas},
we can move from an imperfect state information problem to a perfect information problem by defining a new system where the state is the information vector $\mathcal{I}_t = (o_0,u_0 \dots, o_{t-1}, u_{t-1}, o_t)$, thus reducing the problem to an MDP since $\mathcal{I}_{t+1}$ only depends on $\mathcal{I}_{t}$ and the control $u_t$. We can thus compute the value of this problem, which satisfies a Bellman equation:

\begin{proposition}\label{bellman_solution}
    The optimal value of Problem (\ref{objective}) is given by the Bellman equation:
\begin{equation}\label{bellman_value}
        \psi(\mathcal{I}_t) = \begin{cases}
        v(k_t) - k_t p_{t-1} &\text{if } k_t + \delta_{t-1} \leq m\\
        \max_{u_t \leq k_t} \mathbb{E}_{\omega_t}[\psi(\gamma(\mathcal{I}_t, \overline{\delta}(\omega_t, p_t), u_t) | \mathcal{I}_t] &\text{otherwise}
    \end{cases}
    \end{equation}
    where $\mathcal{I}_t = (o_0, u_0 \dots, o_s = (s, p_s, k_s, \delta_{s-1}), \dots, u_{t-1}, o_t)$ is the information vector at time $t$ and $\gamma(\mathcal{I}_t, \overline{\delta}(\omega_t, p_t), u_t) = (o_0, u_0, \dots, o_{t},u_t,o_{t+1})$ is the updated information vector with $o_{t+1} = (t+1, p_{t+1}, u_t, \overline{\delta}(\omega_t, p_t)))$.
\end{proposition}

\revision{This proposition follows from the results developed in ~\cite{bertsekas,ref_pomdp_fondateur}. Especially, it gives all the necessary steps to model a POMDP with an information vector that takes into account all past rounds. The problem thus becomes a problem with perfect information whose explicit form is given in \Cref{bellman_solution}. To be precise, it is an application of the results in 'Problem with imperfect state information' ~\cite[p. 185–225]{bertsekas}.}

\section{Bellman strategy}\label{Bellman}

\revision{The result of \Cref{bellman_solution} holds regardless of any assumptions made about the opponent. However, its formulation requires maintaining a complete record of the history of the game, which is computationally expensive.  This issue is addressed in the following theorem,  by identifying a condition leading to a more tractable solution.}
\begin{theorem}\label{main_theorem}
    Let $\varphi^v$ be defined on $\mathcal{O}$ by $$\varphi^v(o_0) = \max_{u \leq k_0} \sum_{\delta' = 0}^{(n-1)m} \mathbb{P}(\delta(\betanit) = \delta') \varphi^v(1, \betanit + \Delta P, u, \delta')$$
    
    $\forall t \in \{1, \dots, \overline{R}\}, o_t = (t, p_t, k_t, \delta_{t-1}),$
\begin{equation}\label{simple_bellman}
\!    \varphi^v(o_t) \!=\! \begin{cases}
        v(k_t) - k_tp_{t-1} \quad \hfill\text{if $k_t + \delta_{t-1} \leq m$}\\
        \displaystyle\max_{u_t \leq k_t} \sum_{\delta' \leq \delta_{t-1}}\mathbb{P}(\delta(p_{t}) = \delta' | \delta(p_{t-1}) = \delta_{t-1})\varphi^v(t+1, p_{t+1}, u_t, \delta') \\[-1em]
        \quad\hfill\text{otherwise.}
    \end{cases}
    \end{equation}

    and $\varphi^v(o_{\overline{R}} = (\overline{R}, .)) = 0$.
    
Suppose that $(\delta(p_t) = \overline{\delta}(\omega_t, p_t))_{t\geq 0}$ is a Markov chain. The optimal value given by \Cref{bellman_value} and $\varphi^v$ 
coincide %
:
 $$\forall t \in \{0, \dots, \overline{R}\}, \psi(\mathcal{I}_t) = \varphi^v(o_t)$$
\end{theorem}
In other words, we can find an optimal solution which only depends on the distribution of the opponent's demand, avoiding the recourse to dynamic programming in a belief space. As a matter of fact, the opponent's demand is a sufficient statistics ~\cite{bertsekas} for the optimal value. Indeed, if the demand function is a Markov chain, then, the information vector is also a Markov chain.

\begin{proof}
    Let $\mathcal{I}_t = (o_0 = (0, \betanit, k_0),u_0,  \dots, o_{t-1}, u_{t-1}, o_{t} = (t, p_t, k_t, \delta_{t-1}))$ and $i_{t+1} = (o'_0, u'_0, \dots, o'_{t-1},u'_{t-1},o'_t, u, o)$. Let $u_t$ be an admissible control at time $t$. First, notice that all the information vector $\mathcal{I}_s$ is contained within the vector $\mathcal{I}_{s+1}$ for all round $s$. Therefore, $\mathbb{P}(\mathcal{I}_{t+1} = i_{t+1} | \mathcal{I}_0, \dots, \mathcal{I}_t, u_t) = \mathbb{P}(\mathcal{I}_{t+1} = i_{t+1} | \mathcal{I}_t, u_t)$. The dynamics of the sequence of information vectors only depends on the last information and on the control, $(\mathcal{I}_t)_{0 \leq t \leq \overline{R}}$ is indeed a Markov chain.

    Moreover $\mathbb{P}(o_{t+1} = o | \mathcal{I}_0, \dots, \mathcal{I}_t, u_t) = \mathbb{P}(o_{t+1} = o | \mathcal{I}_t, u_t)$. Thus,
    \begin{align*}
        \mathbb{P}(o_{t+1} = o | \mathcal{I}_t, u_t) &= \mathbb{P}(o_{t+1} = (t+1, p_{t+1}, u, \delta') | o_0, u_0, \dots, o_t, u_t)\\
        &= \mathbb{P}(\overline{\delta}(\omega_0, p_{t}) = \delta' | \delta(\betanit) = \delta_0, \dots, \delta(p_{t-1})=\delta_{t-1})\mathbf{1}(u_t = u)\\
        &= \mathbb{P}(\delta(p_{t}) = \delta' | \delta(\betanit) = \delta_0, \dots, \delta(p_{t-1})=\delta_{t-1})\mathbf{1}(u_t = u)\\
        &= \mathbb{P}(\delta(p_{t}) = \delta' |\delta(p_{t-1}) = \delta_{t-1})\mathbf{1}(u_t = u)
    \end{align*}

    Lastly, the reward of the problem at time $t$ solely depends on $o_t$. Indeed, in the auction, the player receives $v(k_t) - k_tp_{t-1}$ when $k_t + \delta(p_{t-1}) \leq m$ and $0$ otherwise. This shows that $(o_t)_{0 \leq t \leq \overline{R}}$ is a sufficient statistics.
\end{proof}

This leads to a practical algorithm to decide bids at each round: the strategy is to take the $\argmax$ at each round $t$ in \Cref{simple_bellman}. 

\revision{It should be noted that we did not make direct usage of the SB behavior of the opponent. Indeed, the result is true for any demand that is a Markov chain and depends deterministically on $\omega$, the choice of nature and the price. In other words, the results holds for any demand of the opponent that solely depends on their valuation under the assumption that it is a Markov chain. Thereafter, we will focus specifically on the SB strategy which will be assumed to be a Markov chain, and the Bellman strategy the algorithm provides. The following definition recalls the formula for the SB strategy and defines the Bellman strategy.}

\begin{definition}\label{bellman_strat}
    Let $v$ be a player's valuation.
    Let $\nu$ be a \textbf{random} super-player's valuation. We consider the price $p_{t} = \betanit + t\Delta P$.

    We denote by
    \[ \sigma^\nu(t)= \min\big( \argmax_{u \leq (n-1)m}\{\nu(u) - u p_{t} \} \big)
    \]
    the bid of a super SB-player of valuation $\nu$.

    Under the assumptions of \Cref{main_theorem}, we define for  $k + \delta >m$ and $t\leq\overline{R}$ \[ \beta^v(t,k,\delta) = \min\big(\argmax_{u \leq k} Q^v(t, \delta, u)\big)\]
    where
   $Q^v(t, \delta, u) = \sum_{\delta'} \mathbb{P}^t_{\delta'| \delta} \varphi^v(t+1, p_{t+1},u, \delta')$ and $\mathbb{P}^t_{\delta'| \delta} = \mathbb{P}(\sigma^\nu(t+1) = \delta' | \sigma^\nu(t) = \delta)$.
    
    We call Bellman strategy the strategy that consists in bidding $\beta^v(t, k, \delta)$ at round $t$, when the last demand of the player is $k$ and their opponent's last demand is $\delta$. Note that if $k + \delta \leq m$, the auction has terminated so there is no need for a strategy. Similarly, if $t>\overline{R}$, $\beta(t,k,\delta) = 0$.
\end{definition}

\revision{\begin{remark}
    Note that if $(p_t)$ follows the dynamic of the auction, $\forall t \geq 0,\sigma^\nu(t) = \overline{\delta}(\omega_t, p_t)$ where $(w_t)$ is the choice of nature giving a super-player's valuation $\nu$.
\end{remark}}

\begin{remark}
    This definition can be adapted to model a super Bellman-player by modifying the state space, observation space and control space of the POMDP.
\end{remark}

\revision{We have thus shown that if the super-player's demand is a Markov Chain, then we can easily compute the optimal strategy of the player. In the rest of the paper, we will use the more practical form of the optimal solution. Therefore, we introduce the following assumption and show that it holds for reasonable choices of distribution for valuation $\nu$.}
\begin{assumption}[Markov property (MP)]\label{MP}
    $(\sigma^\nu(t))_{0 \leq t \leq \overline{R}}$ is a Markov chain.
\end{assumption}

Let us show some properties of this strategy.

\begin{lemma}\label{lemma1}
A player following the Bellman strategy will always place bids lower than those they would have placed under the SB strategy. Specifically, for a given valuation $v$, we have
    $$\forall t \in \{0, \dots, \overline{R}\}, \forall k \in \{0, \dots, m\}, \forall \delta \in \{0, \dots, (n-1)m\}, {\beta}^v(t, k, \delta) \leq \sigma^v(t).$$
\end{lemma}

In order to prove this result, we show that the strategy that consists in bidding the minimum between the Bellman bid and the SB bid at time $t$ gives the same expected gain as the Bellman strategy. A more detailed proof can be found in \Cref{Appendix}.

We also have a natural upper bound for the expected utility of a Bellman-player at information set $\mathcal{I}_t$.

\begin{proposition}\label{upper_bound}
    Let $o_t$ be an observation at time $t$ such that $k$ is the player's last demand. Then, the value of the POMDP can be bounded as follows: 
    $$\varphi^v(o_t) \leq v(\sigma^v(t-1)\wedge k) - (\sigma^v(t-1)\wedge k)p_{t-1}$$
\end{proposition}

\begin{proof}
We prove the proposition by induction.
    \begin{itemize}
        \item[$\bullet$] Let $(k,\delta) \in \{0, \dots, m\}\times\{0, \dots, (n-1)m\}$.
        
        For $t = \overline{R}$, $\varphi^v(\overline{R}, p_{\overline{R}}, k,\delta) = 0 \leq v(\sigma^v(\overline{R}-1)\wedge k) - (\sigma^v(\overline{R}-1)\wedge k)p_{\overline{R}-1}$.
       
        \item[$\bullet$] Let $(t, k, \delta) \in \{0, \dots, \overline{R}-1\} \times \{0, \dots, m\}\times\{0, \dots, (n-1)m\}$. 
        
        Suppose $\forall (u, \delta') \in \{0, \dots, m\}\times\{0, \dots, (n-1)m\}, \varphi^v(t+1,p_{t+1}, u, \delta') \leq v(\sigma^v(t)\wedge u) - (\sigma^v(t)\wedge u) p_t$.
        If $k + \delta \leq m$ then $\varphi^v(t,p_t, k,\delta) = v(k) - kp_{t-1} \leq v(\sigma^v(t-1)\wedge k) - (\sigma^v(t-1)\wedge k)p_{t-1}$.

        Otherwise,
        \begin{align*}
            \varphi^v(t,p_t, k,\delta) &= \max_{u \leq k} \sum_{\delta'} \mathbb{P}^t_{\delta'|\delta}\varphi^v(t+1, p_{t+1}, u, \delta')\\
            &\leq \max_{u \leq k} \sum_{\delta'} \mathbb{P}^t_{\delta'|\delta}(v(\sigma^v(t)\wedge u) - (\sigma^v(t)\wedge u) p_t)\\
            &= v(\sigma^v(t)\wedge k) - (\sigma^v(t)\wedge k)p_t
        \end{align*}
        Lastly, $v(\sigma^v(t)\wedge k) - (\sigma^v(t)\wedge k)p_t \leq v(\sigma^v(t)\wedge k) - (\sigma^v(t)\wedge k)p_{t-1} \leq v(\sigma^v(t-1)\wedge k) - (\sigma^v(t-1)\wedge k)p_{t-1}$, since $\sigma^v(t)\wedge k \leq \sigma^v(t-1)\wedge k \leq \sigma^v(t-1)$ and $u\mapsto v(u) - up_{t-1}$ is non-decreasing on $\{0, \dots, \sigma^v(t-1)\}$.
    \end{itemize}
\end{proof}

\begin{proposition}\label{undersigma}
    Let $g^v$ defined for all $t \in \{1, \dots, \overline{R}\}, k \in \{0, \dots, \sigma(t-1)\}, \delta \in \{0, \dots, (n-1)m\}$ as $g^v(0,k_0,.) \equiv \max_{u \leq k_0} \sum_{\delta'}\mathbb{P}(\sigma^\nu(1) = \delta')g^v(1,u,\delta')$ and for $t > 0$:
    \[ g^v:(t, k,\delta) \mapsto
    \begin{cases}
        v(k) - kp_{t-1} &\text{if } k + \delta \leq m\\
        \max_{u \leq k \wedge \sigma^v(t) } \sum_{\delta'} \mathbb{P}^t_{\delta'| \delta} g^v(t+1, u, \delta') &\text{if } t< \overline{R} \text{ and } k + \delta > m\\
        0 &\text{otherwise}
    \end{cases} \]
    Then
    $\forall k \in \{0, \dots, m\},\forall \delta \in \{0, \dots, (n-1)m\}, \varphi^v(0, \betanit, k) = g^v(0, k, \delta)$ and
    $\forall (t, k, \delta) \in \{1, \dots, \overline{R}\} \times \{0, \dots, m\}\times \{0, \dots, (n-1)m\},\varphi^v(t,p_t, k,\delta) \leq g^v(t,k \wedge \sigma^v(t-1),\delta) \leq \varphi^v(t, p_t, k \wedge \sigma^v(t-1), \delta)$.
\end{proposition}
Proposition \ref{undersigma} along with proposition \ref{lemma1} shows that it suffices to consider bids in $\{0, \dots, \sigma^v(t-1)\}$ at round $t$ for a Bellman-player. In this case, $\varphi^v(t,p_t, k,\delta) = g^v(t,k,\delta)$.
Such results help reduce the complexity of the Bellman strategy as we reduce the number of scenarios to consider to compute the optimal move to maximize one's expected utility.

\begin{proof}
    We immediately have that $\varphi^v(t, p_t, k\wedge\sigma^v(t-1), \delta) \geq g^v(t, k\wedge\sigma^v(t-1), \delta)$.
    Let us show by induction the left inequality:
    \begin{enumerate}
        \item If $t = \overline{R}$, let $(k,\delta) \in \{0, \dots, m\}\times\{0, \dots, (n-1)m\}$.
        \begin{itemize}
            \item Either $\sigma^v(\overline{R}-1)+\delta > m$ and $k + \delta > m$, in which case $(k\wedge \sigma^v(\overline{R}-1))+\delta > m$.
            Hence, $\varphi^v(\overline{R},p_{\overline{R}}, k,\delta) = 0 = g^v(\overline{R},k\wedge\sigma^v(\overline{R}-1),\delta)$.
            \item Or, we have the case $k + \delta \leq m$ and $\sigma(\overline{R}-1) + \delta > m$ or the case $\sigma^v(\overline{R}-1) + \delta \leq m$ and $k + \delta > m$ which both yield to $(k \wedge \sigma^v(\overline{R}-1)) + \delta \leq m$. Therefore, $g^v(\overline{R},k\wedge\sigma^v(\overline{R}-1),\delta) = v(\sigma(\overline{R}-1)) - \sigma^v(\overline{R}-1)p_{\overline{R}-1} \geq 0 = \varphi^v(\overline{R},p_{\overline{R}}, k,\delta)$.
        \end{itemize}
        \item Let $t \in \{0, \dots, \overline{R}-1\}$ and suppose that $\forall (u, \delta') \in \{0, \dots, m\}\times\{0, \dots, (n-1)m\}, \varphi^v(t+1,p_{t+1}, u,\delta') \leq g^v(t+1,u \wedge \sigma^v(t),\delta')$.
        \begin{itemize}
            \item If $k + \delta \leq m$, then \begin{align*}
            \varphi^v(t,p_t, k,\delta) &= v(k) - kp_{t-1}\\ &\leq v(k\wedge\sigma^v(t-1)) - (k\wedge\sigma^v(t-1))p_{t-1}\\
            &= g^v(t,k\wedge \sigma^v(t-1),\delta)
        \end{align*}
            \item If $k + \delta > m$, then
            \begin{align*}
                \varphi^v(t,p_t, k,\delta) &= \max_{u \leq k} \sum_{\delta'}\mathbb{P}^t_{\delta'|\delta}\varphi^v(t+1,p_{t+1},u,\delta') \\
                &\leq \max_{u \leq k} \sum_{\delta'}\mathbb{P}^t_{\delta'|\delta}g^v(t+1,u\wedge\sigma^v(t),\delta') \\
                &= \max_{u \leq k\wedge\sigma^v(t)} \sum_{\delta'}\mathbb{P}^t_{\delta'|\delta}g^v(t+1,u,\delta') \\
                &= \max_{u \leq (k\wedge\sigma^v(t))\wedge\sigma^v(t-1)} \sum_{\delta'}\mathbb{P}^t_{\delta'|\delta}g^v(t+1,u,\delta') \\
                &= g^v(t,k\wedge\sigma^v(t-1),\delta)
            \end{align*}
        \end{itemize}
    \end{enumerate}
\end{proof}

Lastly, we can observe that the utility decreases with $t$. Therefore, one may have the intuition that if the opponent decreases their demand, one's expected utility might increase as it is possible to stop the auction earlier. The following proposition shows this result, however, we need an assumption to prove it. 

\begin{assumption}[Stochastic monotonicity (SM)]\label{SM}
    $\forall s \in \{0, \dots, (n-1)m\}$, $\forall \delta \in \{0, \dots, (n-1)m-1\}$,$\sum_{\delta' \geq s} \mathbb{P}^t_{\delta'|\delta+1} \geq \sum_{\delta' \geq s} \mathbb{P}^t_{\delta'|\delta}$.
\end{assumption}

This assumption means that there is more chance for a SB demand to be high when the previous demand was high itself. We give in \Cref{Models} examples of probabilistic models in which this assumption is satisfied.

\begin{proposition}\label{decreasingdelta}
    Suppose SM (\Cref{SM}). Then, the value of the problem is non-increasing with $\delta$, i.e.\ 
    $\forall t \in \{1, \dots, \overline{R}\}$,$\forall k \in \{0, \dots, \sigma^v(t-1)\}$,$\forall \delta \in \{0, \dots, (n-1)m-1\}$, $\varphi^v(t,p_t, k, \delta) \geq \varphi^v(t,p_t, k, \delta+1)$.
\end{proposition}

\begin{proof}
    Let $t \in \{1, \dots, \overline{R}\}$, $k \in \{0, \dots, \sigma^v(t-1)\}, \delta \in \{0, \dots, (n-1)m-1\}$.

    \begin{enumerate}
        \item If $k + \delta + 1 \leq m$ then 
        \begin{align*}
            \varphi^v(t, p_t,k, \delta+1) - \varphi^v(t, p_t, k, \delta) = (v(k) - k p_{t-1}) - (v(k) - kp_{t-1}) = 0
        \end{align*}
        \item  If $k + \delta = m$ then
        \begin{align*}
            \varphi^v(t, p_t,k, \delta+1) - \varphi^v(t, p_t, k, \delta) &= \varphi^v(t, p_t, k, \delta+1) - (v(k) - kp_{t-1})\\
            &\leq  (v(k) - k p_{t-1}) - (v(k) - kp_{t-1}) = 0
        \end{align*}
        \item Lastly for the case $k + \delta > m$, Let us proceed by induction.

        The initialization is immediate at $t=\overline{R}$.

        Now suppose $\forall u \in \{0, \dots, \sigma^v(t)\}, \forall \delta' \in \{0, \dots, (n-1)m-1\}, \varphi^v(t+1,p_{t+1}, u, \delta'+1) \leq \varphi^v(t+1, p_{t+1}, u, \delta')$.
        Per Proposition \ref{undersigma}, $\varphi^v(t,p_t,k,\delta) = g^v(t,k,\delta)= \max_{u \leq k\wedge\sigma^v(t)} \sum_{\delta' = 0}^{(n-1)m} \mathbb{P}^t_{\delta'|\delta}\varphi^v(t+1,p_{t+1},u,\delta')$.
        
        Following an equivalent of stochastic ordering (Condition B' in ~\cite{stochastic_order}), for all $f$ non-increasing function, $\hat{f}: \delta \mapsto \sum_{\delta' = 0}^{(n-1)m} \mathbb{P}^t_{\delta'|\delta}f(\delta')$ is non-decreasing.
        
        Hence, $\delta \mapsto \sum_{\delta' = 0}^{(n-1)m}  \mathbb{P}^t_{\delta'|\delta} \varphi^v(t+1,p_{t+1},u,\delta')$ is non-increasing, thus $\delta \mapsto \varphi^v(t,p_t,k,\delta)$ is non-increasing as maximum of non-increasing functions.
    \end{enumerate}
\end{proof}

 \section{Guarantees}\label{robust}

There is no guarantee that the Bellman strategy would be robust to a change of the opponent's strategy. Thus, one may wonder what happens if the opponent deviates from SB, for instance, what happens when both players use the Bellman strategy. In order words, the Bellman-player is optimal against a SB-player, but what happens when both players think they face a SB-player and play accordingly?

\begin{definition}\label{gain_def}
    Let $v$ be a deterministic valuation and $\nu$ be a random valuation.

    Let $s^v$ (resp. $\xi^\nu$) be a strategy of the $v$-player (resp. $\nu$-player).
    Let $(k_0, \dots, k_{t-1})$ (resp. $(\delta_0, \dots, \delta_{t-1}))$ be the bids of the $v$-player (resp. $\nu$-player) before $t$.
    We define $h_t := (k_0,\delta_0, \dots, k_{t-1}, \delta_{t-1})$ the past history of bids.

    We set $k_t := s^v(t, h_t)$ and $\delta_t := \xi^\nu(t, h_t)$. Note that $\delta_t$ is random and depends on $\nu$, $\xi^\nu$ and all the past information. Therefore, it defines a random vector $h_{t+1} = (h_t, k_t, \delta_t)$.

    The expected utility of the $s^v$-player at round $t$ is given by:
    $$ G_{t, s^v, \xi^\nu}(t, h_t) := 
    \begin{cases}
    v(k_{t-1}) - k_{t-1}p_{t-1} &\text{ if } k_{t-1} + \delta_{t-1} \leq m\\
    \mathbb{E}_{\delta_t}[G_{t+1, s^v, \xi^\nu}(t+1,h_{t+1})] &\text{ if } k_{t-1}+\delta_{t-1} > m \text{ and } t \leq \overline{R}\\
    0 &\text{ otherwise}
    \end{cases}$$
\end{definition}

\begin{remark}
\revision{Note that this definition gives a general expression for the gain of a player. It allows for a deeper comprehension of the interaction between players. Indeed, with this expression, we can compare the expected utility obtained by a player by specifying their strategy and the strategy of the opponent. It allows us to focus on the expected utility of the player rather than their strategy as it becomes an input. Since the previous sections study the case of a Bellman player against a SB player, it should be noted that $G_{t, \beta^v, \sigma^\nu}(t, h_t) = \varphi^v(t,p_t, k_{t-1},\delta_{t-1})$.}
\end{remark}

\revision{The previous remark helps us rewrite the theorems from the previous section with the notation of \Cref{gain_def}. Hence, the following theorem is a reformulation of \Cref{main_theorem} making use of the general expression $G$:}

\begin{theorem}\label{theorem1}
    Let $v$ and $\nu$ as in the definition above and suppose MP (\Cref{MP}). 
    
    Then, $\forall t \in \{0, \dots, \overline{R}\}, \forall (k_r)_{0 \leq r \leq t-1}, \forall (\delta_r)_{0 \leq r \leq t-1}, \forall s^v \text{ strategy},$
    $$G_{t, \beta^v, \sigma^\nu}(t,h_{t}) \geq G_{t, s^v, \sigma^\nu}(t,h_t)$$
    
    In order words, the Bellman strategy maximizes the expected gain of the player against a SB-player.
\end{theorem}

\begin{theorem}\label{theorem2}
Let $v$ and $\nu$ as defined in theorem \ref{theorem1}. 

Suppose SM (\Cref{SM}).Then, the expected gain of a Bellman-player against a Bellman opponent is higher than the expected gain obtained when the opponent is an SB-player, i.e.\

For $t \in \{0, \dots, \overline{R}\}$, $(k_r)_{0\leq r<t} \in \{0, \dots, m\}^t$ and $(\delta_r)_{0\leq r<t} \in \{0, \dots, (n-1)m\}^t$ two sequences of bids and $h_t = (k_0, \delta_0, \dots, k_{t-1}, \delta_{t-1})$ the past history, we have
    $$G_{t, \beta^v, \beta^\nu}(t,h_{t}) \geq G_{t, \beta^v, \sigma^\nu}(t,h_{t})$$
\end{theorem}

In this case, the opponent playing $\beta^\nu$ is an aggregation of Bellman-players.  

\begin{proof}
    Let us prove it by induction. 
    
    For $t = \overline{R}$, $G_{t, \beta^v, \beta^\nu}(t,h_{t}) = 0 = G_{t, \beta^v, \sigma^\nu}(t,h_t)$.
    
    Let $t \in \{0, \dots, \overline{R}\}$. Let $(k_r)_{r<t}$ and $(\delta_r)_{r<t}$ be two sequences of bids. Suppose $\forall k_t \leq k_{t-1}$,$\forall \delta_t \leq \delta_{t-1}$,$G_{t+1, \beta^v, \beta^\nu}(t,h_t,k_t, \delta_{t}) \geq G_{t, \beta^v, \sigma^\nu}(t,h_t,k_t, \delta_{t})$.

    If $k_{t-1} + \delta_{t-1} \leq m$ then $G_{t, \beta^v, \beta^\nu}(t,h_t) = v(k_{t-1}) - k_{t-1}p_{t-2} = G_{t, \beta^v, \sigma^\nu}(t,h_t)$. otherwise,
\begin{align*}
                G_{t, \beta^v, \beta^\nu}(t,h_t) &= \mathbb{E}[G_{t+1, \beta^v, \beta^\nu}(t+1,h_t,\beta^v(t, k_{t-1},\delta_{t-1}),\beta^\nu(t, \delta_{t-1},k_{t-1}))]\\
                &\geq \mathbb{E}[G_{t+1, \beta^v, \sigma^\nu}(t+1,h_t,\beta^v(t, k_{t-1},\delta_{t-1}),\beta^\nu(t, \delta_{t-1},k_{t-1}))]\\
                &\; \hfill\text{  by induction hypothesis}\\
                &= \mathbb{E}[\varphi^v(t+1,p_{t+1}, \beta^v(t, k_{t-1}, \delta_{t-1}), \beta^\nu(t, \delta_{t-1},k_{t-1})]
            \end{align*}
$\beta^v(t, k_{t-1},\delta_{t-1}) \leq \sigma^v(t)$ and $\beta^\nu(t, \delta_{t-1}, k_{t-1}) \leq \sigma^\nu(t) (*)$ thus we can apply proposition \ref{decreasingdelta}.
\begin{align*}
                G_{t, \beta^v, \beta^\nu}(t,h_t) &\geq \mathbb{E}[\varphi^v(t+1,p_{t+1}, \beta^v(t, k_{t-1}, \delta_{t-1}), \beta^\nu(t, \delta_{t-1},k_{t-1})]\\
                &\geq \mathbb{E}[\varphi^v(t+1,p_{t+1}, \beta^v(t, k_{t-1}, \delta_{t-1}), \sigma^\nu(t))]\\
                &= G_{t, \beta^v, \sigma^\nu}(t,h_t)
            \end{align*}
\end{proof}

Intuitively, a Bellman-player would try to end the auction as soon as possible as long as the cost of stopping the auction outweighs the cost of rising the price. SB, as a myopic strategy enters the auction greedily whereas Bellman estimates when it is no longer profitable to be greedy. Therefore, when both players play the Bellman strategy, they both expect their opponent to be greedy, thus they reduce their demands in order to make the auction stop earlier. This behavior leads to an increase of the player's expected utility.

This result is actually more general as we do not make use of the fact that the opponent plays Bellman but only that their demand is less than the demand they would have made had they played SB.

\begin{theorem}\label{theorem3}
    Let $v$ and $\nu$ as defined in theorem \ref{theorem1} and suppose SM. Let $t \in \{0, \dots, \overline{R}\}$, $(k_r)_{r<t} \in \{0, \dots, m\}^t$, $(\delta_r)_{r<t} \in \{0, \dots, (n-1)m\}^t$ two sequences of bids and $h_t = (k_0, \delta_0, \dots, k_{t-1}, \delta_{t-1})$. Let $\xi^\nu$ be a strategy such that $\forall t \in \{0, \dots, \overline{R}\}, \xi^\nu(t, h_t) \leq \sigma^\nu(t)$. Then, the expected gain of a Bellman-player against $\xi^\nu$ is bounded from below by the expected gain against a SB-player of valuation $\nu$.

    In order words,  
    $$G_{t, \beta^v, \xi^\nu}(t,h_t) \geq G_{t, \beta^v, \sigma^\nu}(t,h_t)$$
\end{theorem}
\revision{
\begin{proof}
    The proof is identical to the one of \Cref{theorem2}.
    Indeed, in the latter, the fact that the opponent plays Bellman only rises in step $(*)$ of the proof in order to apply \Cref{decreasingdelta}. Thus, the result remains true for any strategy such as \Cref{decreasingdelta} is applicable, i.e.\ that are less or equal than SB.
\end{proof}
}

Regardless of the opponent's strategy, as long as the player plays Bellman, they can guarantee an expected gain of at least the expected gain they would have had had they played against a super SB-player.

\section{Probabilistic models}\label{Models}
\newcommand{\player}{\text{player}}
\newcommand{\opp}{\text{opp}}
\newcommand{\Zmax}{Z_{\max}}

In this section, we introduce a framework for our experiments. Mainly, this means to choose a distribution of the $(Z_j)_{1 \leq j \leq (n-1)m}$ defined in the beginning of \Cref{pomdp_model} in order to define the super-player's valuation.

Ideally, such distribution should respect the two assumptions of our theorems:

\begin{enumerate}
    \item[(i)] Markov property (\Cref{MP}).
    \item[(ii)] Stochastic Monotonicity (\Cref{SM}).
\end{enumerate}

\subsection{Uniform case}
The first intuition is to uniformly draw the $(Z^i_j)$, i.e.\ drawing uniformly the slopes of each opponent and then, reordering them in order to form the super-player's valuation. This model would embody the idea that the opponents would have a budget $\Zmax$ that could not be exceeded.

\begin{definition}
    Let $\Zmax > 0$.
    Let $(U^i_j)_{\substack{1 \leq i \leq n-1 \\ 1 \leq j \leq m}} \sim \mathcal{U}[0, 1]$ iid. For $i \in \{1, \dots, n-1\}$, Let $\forall j \in \{1, \dots, m\}, \tilde{Z}^i_j = \Zmax U^i_j$ and define $Z^i_j = \tilde{Z}^i_{(m-j+1)}$ the $(m-j+1)^{th}$ order statistics of $(Z^i_1, \dots, Z^i_m)$. Let $(Z_1, \dots, Z_{(n-1)m})$ be a reordering of $(Z^1_1, \dots, Z^1_m, \dots, Z^i_j, \dots, Z^{n-1}_m)$ such as $Z_1 \geq \dots \geq Z_{(n-1)m}$. 
    
    We call an opponent of valuation $V: k \in \{0, \dots, (n-1)m\} \mapsto\sum_{j = 1}^k Z_j$ uniformly distributed of parameter $\Zmax$.
\end{definition}

\begin{proposition}\label{uniform}
    A uniformly distributed super SB-player satisfies MP and SM.
\end{proposition}

We rely on the following lemma for the first part of the proof:
    \begin{lemma}\label{uniform_lemma}
        The density of $(Z_1, \dots, Z_{(n-1)m})$ is given by 
        
        $z \in \mathbb{R}^{(n-1)m} \mapsto \frac{((n-1)m)!}{Z_{\max}^{(n-1)m}} \mathbf{1}_{z_1 > \dots > z_{(n-1)m}}(z) \mathbf{1}_{z \in [0, Z_{\max}]^{(n-1)m}}(z)$.
    \end{lemma}

    \Cref{uniform_lemma} is proven in \Cref{Appendix}.

\begin{proof} \textit{First part}
    Let $\Zmax > 0$ and $(Z_j)$ be a sequence of random variables such that an opponent of valuation $V: k \in \{0, \dots, (n-1)m\} \mapsto\sum_{j = 1}^k Z_j$ is uniformly distributed. Let $(\delta_0, \dots, \delta_t, \delta_{t+1})$ be a sequence of bid. Note that $\mathbb{P}(\delta(p_0) = \delta_0, \dots, \delta(p_t) = \delta_t) = 0$ if the sequence does not satisfy $\delta_0 \geq \dots \geq \delta_{t+1}$. Therefore, we suppose the sequence is non-increasing. We denote by $E_{t+1}$ the event $\{ \delta(p_0) = \delta_0, \dots, \delta(p_{t+1}) = \delta_{t+1} \}$, so that 
\begin{align*}
    E_{t+1} =& \{ \delta(p_0) = \delta_0, \dots, \delta(p_{t+1}) = \delta_{t+1} \}\\
        =& \big\{\sum_{j = 1}^{(n-1)m}\mathbf{1}(Z_j > p_0) = \delta_0, \dots, \sum_{j = 1}^{(n-1)m}\mathbf{1}(Z_j > p_{t+1}) = \delta_t \big\}\\
        =& \{0 < Z_{(n-1)m} \leq p_0 \}\cap \dots\cap \{0 < Z_{\delta_0 + 1} \leq p_0\}\cap\{p_{t-1} < Z_{\delta_t + 1} \leq p_t \}\\&\cap\{ p_t < Z_{\delta_t} \leq p_{t+1}\}\cap\dots\cap\{p_t < Z_{\delta_{t+1}+1} \leq p_{t+1} \}\\
        &\cap\{ p_{t+1} < Z_{\delta_{t+1}} \leq \Zmax \}\cap \dots \cap \{ p_{t+1} < Z_1 \leq \Zmax \}\\
        =& \{ Z \in A \}
    \end{align*}
where $A = (p_{t+1}, Z_{\max}]^{\delta_{t+1}} \times \mybigtimes_{s = 0}^t (p_{s}, p_{s+1} ]^{\delta_s - \delta_{s+1}} \times  (0, p_{t} ]^{(n-1)m - \delta_0}$.

\Cref{uniform_lemma} gives
$$\mathbb{P}(E_{t+1}) = \frac{((n-1)m)!}{Z_{\max}^{(n-1)m}}\int_{A} \mathbf{1}_{z_1 \geq \dots \geq z_{(n-1)m}}(\mathbf{z}) d\mathbf{z}$$

Let $\mathbf{z} = (z_1, \dots, z_{(n-1)m}) \in A$. We have
\begin{align*}\mathbf{1}_{z_1 \geq \dots \geq z_{(n-1)m}}(\mathbf{z}) =& \mathbf{1}_{z_1 \geq \dots \geq z_{\delta_{t+1}}}(\mathbf{z})\mathbf{1}_{z_{\delta_{t+1}} \geq z_{\delta_{t+1} + 1}}(\mathbf{z})\mathbf{1}_{z_{\delta_{t+1} + 1} \geq \dots \geq z_{(n-1)m}}(\mathbf{z}) \\
=& \mathbf{1}_{z_1 \geq \dots \geq z_{\delta_{t+1}}}(\mathbf{z})\mathbf{1}_{z_{\delta_{t+1} + 1} \geq \dots \geq z_{(n-1)m}}(\mathbf{z})
\end{align*}
since $z_{\delta_{t+1}} \in (p_{t+1}, \Zmax]$ and $z_{\delta_{t+1}+1} \in (p_{t}, p_{t+1}]$. Let us denote by $z^{(i-j)} = (z_{i+1}, \dots, z_j)$. By iterating the above computation, we obtain 
\begin{align*}
    \mathbb{P}(E_{t+1})=&\frac{((n-1)m)!}{Z_{\max}^{(n-1)m}}\big(\int_{(p_{t+1}, Z_{\max}]^{\delta_{t+1}}} \mathbf{1}_{\{z_1 \geq \dots \geq z_{\delta_{t+1}}\}}(z^{(-\delta_{t+1})}) d(z^{(-\delta_{t+1})})\big)\\
    \times& \prod_{s = 0}^{t} \big(\int_{(p_{s}, p_{s+1}]^{\delta_s - \delta_{s+1}}} \mathbf{1}_{\{z_{\delta_{s+1} +1} \geq \dots \geq z_{\delta_{s}}\}}(z^{(\delta_{s+1} - \delta_{\delta_s})}) d(z^{(\delta_{s+1} - \delta_{\delta_s})}) \big)\\
    \times& \big(\int_{(0, p_0]^{(n-1)m-\delta_{0}}} \mathbf{1}_{\{z_{\delta_{0} +1} \geq \dots \geq z_{(n-1)m} \}}(z^{(\delta_{0} - (n-1)m)}) d(z^{(\delta_{0} - (n-1)m)}) \big)
\end{align*}
Therefore,
\begin{eqnarray*}
\lefteqn{
  \frac{\mathbb{P}(\delta(p_{t+1}) = \delta_{t+1}, \dots, \delta(p_0) = \delta_0)}{\mathbb{P}(\delta(p_{t}) = \delta_{t}, \dots, \delta(p_0) = \delta_0)}} &\\
&= & \frac{\big(\int_{(p_{t+1}, Z_{\max}]^{\delta_{t+1}}} \mathbf{1}_{\{z_1 \geq \dots \geq z_{\delta_{t+1}}\}}(z^{(-\delta_{t+1})}) d(z^{(-\delta_{t+1})})\big)}{\big(\int_{(p_{t}, Z_{\max}]^{\delta_{t}}} \mathbf{1}_{\{z_1 \geq \dots \geq z_{\delta_{t}}\}}(z^{(\delta_{t})}) d(z^{(\delta_{t})})\big)}\\
&&\times \big(\int_{(p_{t}, p_{t+1}]^{\delta_t - \delta_{t+1}}} \mathbf{1}_{\{z_{\delta_{t+1} +1} \geq \dots \geq z_{\delta_{t}}\}}(z^{(\delta_{t+1} - \delta_{\delta_t})}) d(z^{(\delta_{t+1} - \delta_{\delta_t})}) \big)
\end{eqnarray*}
Hence, $\mathbb{P}(\delta(p_{t+1}) = \delta_{t+1}|\delta(p_{t}) = \delta_{t}, \dots, \delta(p_0) = \delta_0)$ depends on $\delta_{t+1}, \delta_t, p_{t+1}$ and $p_t$. $\mathbb{P}(\delta(p_{t+1}) = \delta_{t+1}|\delta(p_{t}) = \delta_{t}, \dots, \delta(p_0) = \delta_0) = \mathbb{P}(\delta(p_{t+1}) = \delta_{t+1}|\delta(p_{t}) = \delta_{t})$.
MP is thus satisfied. \end{proof}

The second part of the proof rely on another lemma:

\begin{lemma}\label{binomial}
    Let $u = \frac{\Zmax - p_{t+1}}{\Zmax - p_t}$.
    If $\Zmax \leq p_{t+1}$, then $\delta(p_{t+1}) = 0$. Otherwise, $\forall \delta \in \{0, \dots, (n-1)m\}, \forall \delta' \leq \delta, \mathbb{P}(\delta(p_{t+1}) = \delta' | \delta(p_{t}) = \delta ) = \binom{\delta}{\delta'} u^{\delta'} (1 - u)^{\delta - \delta'}$
\end{lemma}

\Cref{binomial} is proven in \Cref{Appendix}.

\begin{proof} \textit{Second part}

The law of $\delta(p_{t+1})$ conditionally to $\{\delta(p_t) = \delta_t\}$ is a binomial law. It can be interpreted as the number of successes in a sequence of $\delta$ independent experiments. We fix $(X_j)_{j \in \mathbb{N}}$ iid Bernoulli variable of parameter $u$.

Then, we have $\sum_{\delta' = s}^{(n-1)m}\mathbb{P}(\delta(p_{t+1}) = \delta' | \delta(p_{t}) = \delta ) = \sum_{\delta' = s}^{(n-1)m}\binom{\delta}{\delta'} u^{\delta'} (1 - u)^{\delta - \delta'} = \mathbb{P}(X_1 + \dots + X_\delta \geq s)$, the probability to have at least $s$ successes in $\delta$ experiments.

Since $\{X_1 + \dots + X_\delta \geq s\} \subseteq  \{X_1 + \dots + X_{\delta+1} \geq s\}$, $\sum_{\delta' = s}^{(n-1)m}\mathbb{P}(\delta(p_{t+1}) = \delta' | \delta(p_{t}) = \delta ) \leq \sum_{\delta' = s}^{(n-1)m}\mathbb{P}(\delta(p_{t+1}) = \delta' | \delta(p_{t}) = \delta+1 )$. SM is thus also satisfied.   
\end{proof}
\subsection{Exponential case} The exponential case can be interesting to compute also since it is a traditional framework to model counting processes. Indeed, the demand of a super-opponent is nothing but an inverse counting process: we count the number of slopes of the valuation that are still above the price once it is raised.

\begin{definition}\label{exponential_case}
    Let $\lambda > 0$.
    Let $(Z_j)_{1 \leq j \leq (n-1)m}$ such as $\forall j \in \{1, \dots, (n-1)m\}, Z_j - Z_{j+1} \sim \mathcal{E}(\lambda)$ iid with $Z_{(n-1)m+1} = 0$.
    We call a super SB-player  of valuation $V(k) = \sum_{j = 1}^k Z_j$ an exponential super SB-player of parameter $\lambda$.
\end{definition}
\begin{proposition}
    An exponential super SB-player satisfies MP and SM.
\end{proposition}
\begin{proof} \textit{First part}
    As mentioned, we will show that $\delta(p) = \sum_{j = 1}^{(n-1)m} \mathbf{1}(Z_j > p)$ is closely related to a counting process and more specifically a Poisson process. Let $(T_j)_{j \in \mathbb{N}}$ be defined as follows:
    \begin{itemize}[-]
        \item $T_0 = 0$
        \item $\forall j \in \mathbb{N}, T_{j+1} - T_j \sim \mathcal{E}(\lambda)$ iid
    \end{itemize}
    Then, $N_t := \sup\{n\in \mathbb{N}|T_n \leq t\} = \sum_{j = 0}^{+\infty}\mathbf{1}(T_j \leq t)$ is a Poisson process of parameter $\lambda$ (Chapter 4, ~\cite{LefebvreMario2007ASP}). Note that $T_j$ has the same law as $Z_{m+1-j}$ for $j \in \{1, \dots, m\}$. Hence, since $\delta(p) = \sum_{j = 1}^{(n-1)m} \mathbf{1}(Z_j > p) = (n-1)m - \sum_{j = 1}^{(n-1)m} \mathbf{1}(Z_j \leq p)$, $\delta(p)$ has the same law as $(n-1)m - (N_p \wedge (n-1)m)$. Because $(N_t)$ is a non-decreasing process with integer values, and since $Z_1$ has the same law as $T_m = \inf\{t\in\mathbb{R}_{+}|N_t = (n-1)m\}$, $(N_p \wedge (n-1)m)$ has the same law as $N_{p \wedge Z_1}$. Hence, $\delta(p)$ has the same law as a $f(X_p)$ where $f : x\mapsto (n-1)m - x$ is injective and $X_p = N_{p \wedge Z_1}$ is a Markov chain (since $(N_p)$ is a Markov chain, see Chapter 4 of ~\cite{LefebvreMario2007ASP} and section 7.2. of ~\cite{ash_basic_2008}). Therefore, $\delta(p)$ is a Markov chain. MP is thus satisfied. \end{proof}

    In the second part of the proof, we use the following lemma. Its proof can be found in \Cref{Appendix}.

    \begin{lemma}\label{proba_exp}
    Let $t \in \{0, \dots, \overline{R}\}$, $\delta_t \in \{1, \dots, (n-1)m\}$.
        $\forall \delta_{t+1} \in \{0, \dots, \delta_t\},$\begin{align*}
    \mathbb{P}(\delta(p_{t+1}) = \delta_{t+1} | \delta(p_t) = \delta_t ) &= \begin{cases}
        e^{-\lambda \Delta P}\sum_{j = \delta_t}^{+\infty} \frac{(\lambda \Delta P)^j}{j!} &\text{ if } \delta_{t+1} = 0\\
        e^{-\lambda \Delta P}\frac{(\lambda \Delta P) ^{\delta_t - \delta_{t+1}}}{(\delta_t - \delta_{t+1})!} &\text{ otherwise}
    \end{cases}
\end{align*}
and
    $\mathbb{P}(\delta(p_{t+1}) = \delta_{t+1} | \delta_t = 0) = \mathbf{1}(\delta_{t+1} = 0)$.
    \end{lemma}

\begin{proof} \textit{Second part}
Let $s \in \mathbb{N}$ and $\delta \in \{0, \dots, (n-1)m-1\}$. 
If $s = 0$, then $\sum_{\delta'=s}^{(n-1)m} \mathbb{P}(\delta(p_{t+1}) = \delta' | \delta(p_t) = \delta) = 1 = \sum_{\delta'=s}^{(n-1)m} \mathbb{P}(\delta(p_{t+1}) = \delta' | \delta(p_t) = \delta+1)$.

If $s > 0$, then, if $\delta = 0$, $\sum_{\delta'=s}^{(n-1)m} \mathbb{P}(\delta(p_{t+1}) = \delta' | \delta(p_t) = \delta) = 0 \leq \sum_{\delta'=s}^{(n-1)m} \mathbb{P}(\delta(p_{t+1}) = \delta' | \delta(p_t) = \delta+1)$. 
Otherwise,
\begin{align*}
    \sum_{\delta'=s}^{(n-1)m} \mathbb{P}(\delta(p_{t+1}) = \delta' | \delta(p_t) = \delta) &=\sum_{\delta' = s}^{\delta} \mathbb{P}(\delta(p_{t+1}) = \delta' | \delta(p_t) = \delta)\\
    &= \sum_{\delta' = s}^{\delta} e^{-\lambda \Delta P}\frac{(\lambda \Delta P) ^{\delta - \delta'}}{(\delta - \delta')!}\\
    &= \sum_{\delta' = 0}^{\delta - s}e^{-\lambda \Delta P}\frac{(\lambda \Delta P) ^{\delta'}}{\delta'!}\\
    &\leq \sum_{\delta' = 0}^{\delta+1 - s}e^{-\lambda \Delta P}\frac{(\lambda \Delta P) ^{\delta'}}{\delta'!}\\
    &= \sum_{\delta'=s}^{(n-1)m} \mathbb{P}(\delta(p_{t+1}) = \delta' | \delta(p_t) = \delta+1)
\end{align*}
    All in all, SM is satisfied.
\end{proof}

\begin{remark}
    The exponential setting of \Cref{exponential_case} does not rely on the individual laws of each opponent but rather on the aggregation of the opponents.

    One could wonder if we could model each opponent as a counting process and then aggregate them in such way that (i) and (ii) are satisfied. This would yield to $(Z^i_j)_{\substack{1 \leq i \leq n-1 \\ 1 \leq j \leq m}}$, $(Z_j^i)$ has independent increments of exponential law. Then, $\delta(p) = \sum_{i = 1}^{n-1} \delta_i(p) = (n-1)m - \sum_{i = 1}^{n-1} N^i_{p \wedge Z^i_1}$. However, such a variable is not a Markov chain. The difficulty lies in the fact that for each player, There is a different stopping time, making the state of variable depend on the players that have dropped out the auction.

    A way to circumvent this hardship is to consider that each opponent can make negative demands. Hence, $\delta(p) = (n-1)m - \sum_{i = 1}^{n-1} N^i_{p} \in (-\infty, (n-1)m]$ is a Markov chain. We can interpret this new demand function as follows: once the price is too high, players sell items at the current price rather than buy some. As a result, in a aggregation of two players $A$ and $B$, where $A$ would bid $k$ and $B$ would bid $-1$, then since $A$ wishes to buy $k$ items at the current price, they would buy $1$ from $B$ as they sell one and still ask for $k-1$ items on the auction. The aggregated demand would thus be $k-1$. 
    
    However, in the context of spectrum auctions, such model is unrealistic, that is why we choose to model the aggregated valuation directly. However such interpretation can be useful for modeling stock markets for instance.
\end{remark}

\section{Experimental results}\label{Example}

\Cref{main_theorem} provides a simple algorithm to play an auction optimally against a SB-player, by applying the Bellman strategy. In this section, we quantify numerically the value of this optimum.
We investigate how much the algorithm's expected utility deviates from the utility of a player with perfect information (see the optimization problem \ref{perfect}).
We simulate auctions and evaluate the performance of an agent playing according to the strategy we have outlined. 

\subsection{Simulation setting}

Among the simulations, $\Delta P$ the price increment, $\betanit$ the initial price and $m$ the number of items are fixed. Each simulation is carried out as follows:
\begin{itemize}
    \item We draw a realization of $(Z_j) : z_1, \dots, z_{(n-1)m}$ to define the valuation of the exponential super SB-player. 
    
    Their demand is $\delta(p) = \sum_{j = 1}^{(n-1)m} \mathbf{1}\{z_j > p\}$.
    \item We then run two auctions using two different strategies against the same opponent:
    \begin{enumerate}
        \item For the first auction, we suppose the player has access to the opponent's valuation.
        The player plays with perfect information and obtains an optimal final score $U$, their utility at the end of the auction.
        \item The second strategy is the Bellman strategy defined in \Cref{bellman_strat}. This results in a final score $\hat{U}$.
    \end{enumerate}
    \end{itemize}
After simulating $N$ auctions with the same parameters, we obtain $N$ pairs of scores $(U_1, \hat{U}_1), \dots, (U_N, \hat{U}_N)$.  Those scores can be seen as the realization of two random variables $(U, \hat{U})$ which would give the score of an auction with perfect and imperfect information respectively. We can then estimate and compare the two random variables and the expected gains $\mathbb{E}[U]$ and $\mathbb{E}[\hat{U}]$.

\subsection{Empirical evidence} In order to choose the most realistic parameters, we consider the same parameters as in the French 5G auction: $m = 11$, $\betanit = 70$ and $\Delta P = 3$ (see ~\cite{attribution_freq}).

We model the valuation of the player and the opponent using either the uniform or the exponential models.

\subsubsection{Uniform case}

We have conducted $N = 10,000$ auctions with exponential valuations for $\Zmax \in \{110, 121, 132, 143, 154, 165\}$ in order to mimic different bidding profiles compatible with the 5G auction.
We present the expected gain in the following table:
\begin{table}[H]
\begin{minipage}{.5\linewidth}
\centering
\includegraphics[scale = 0.4]{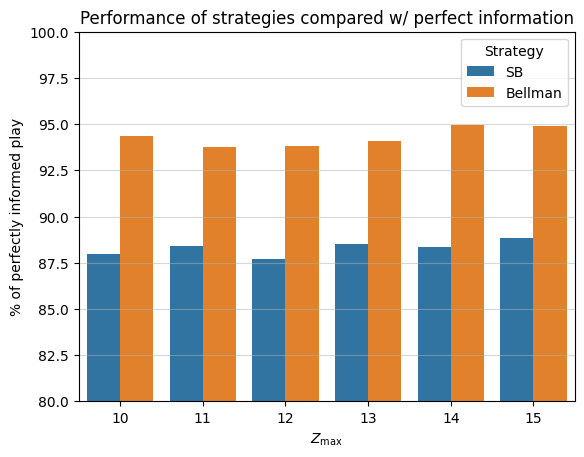}
\end{minipage}
\begin{minipage}{.5\linewidth}
\caption{Empirical expectations}
\centering\footnotesize\setlength{\tabcolsep}{3pt}
    \begin{tabular}{c|c|c|c|c|c|c}
    $\Zmax$ & 110 & 121 & 132 & 143 & 154 & 165 \\
    \hline
    $\mathbb{E}[U]$ & 37 & 40 & 46 & 50 & 54 & 55\\
    \hline
    $\mathbb{E}[\hat{U}]$ & 35 & 38& 43 & 47 & 51 & 52\\
    \hline
    $\mathbb{E}[\hat{U}] / \mathbb{E}[U]$ & 94\% &94\%	&94\%	&94\%	&95\%	&95\%
\end{tabular}
\end{minipage}
\caption{On the graph on the left, we show the expected gain for a player playing either the SB strategy (blue) or the Bellman strategy (orange) against a super SB-player. The $100\%$ on the graph represents the expected utility of a perfectly informed player against the same super SB-player.}
\end{table}

We observe two key points:
\begin{enumerate}
    \item The SB strategy is quite efficient against a SB super-player with expected gain of $88\%$ the maximal expected gain.
    \item However, it is outperformed by the Bellman strategy which is even closer to the maximal expected gain.
\end{enumerate}

\begin{table}[H]
\begin{minipage}{.5\linewidth}
\centering
\includegraphics[scale = 0.4]{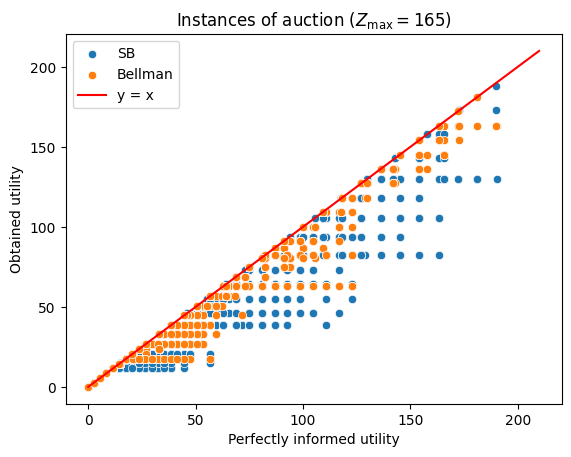}
\end{minipage}
\begin{minipage}{.5\linewidth}
  \caption{Estimation of $\mathbb{P}(U = \hat{U})$}\footnotesize\setlength{\tabcolsep}{3pt}
\centering
    \begin{tabular}{c|c|c|c|c|c|c}
    $\Zmax$ & 110 & 121 & 132 & 143 & 154 & 165 \\
    \hline
    $\mathbb{P}(U = \hat{U})$ & 56\%	&61\%	&60\%	&59\%&	62\%	&62\%
\end{tabular}
\end{minipage}
\caption{On the left, each point represents the result of an auction: on the x-axis we have the utility $U$ and on the y-axis the utility $\hat{U}$. The color of a point indicates which strategy the player is playing. More than half of the sampled auctions playing Bellman yields the same utility as the optimal perfectly informed play.}
\end{table}%

\subsection{Exponential case}

Furthermore, we have conducted $N = 10,000$ auctions with exponential valuations for $\lambda \in \{10, 11, 12, 13, 14, 15\}$ in order to mimic different bidding profiles compatible with the 5G auction. We computed the same figures as the uniform case. We observe that the conclusions are the same.

\begin{table}[H]
\begin{minipage}{.5\linewidth}
\centering
\includegraphics[scale = 0.4]{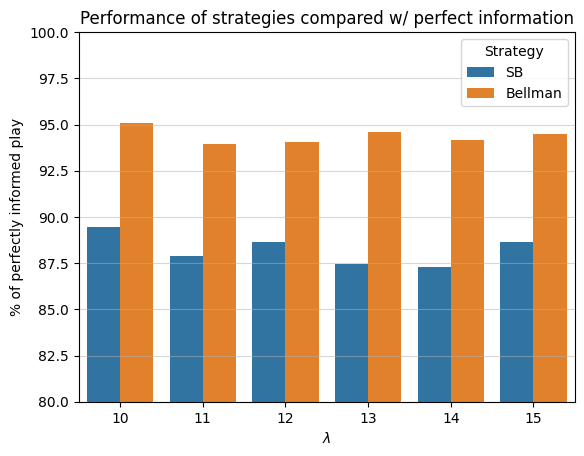}
\end{minipage}
\begin{minipage}{.5\linewidth}
\caption{Empirical expectations}
\centering\footnotesize\setlength{\tabcolsep}{3pt}
    \begin{tabular}{c|c|c|c|c|c|c}
    $\lambda$ & 10 & 11 & 12 & 13 & 14 & 15 \\
    \hline
    $\mathbb{E}[U]$ & 50 & 59 & 66 & 72 & 78 & 83\\
    \hline
    $\mathbb{E}[\hat{U}]$ & 46 & 55& 62 & 68 & 73 & 78\\
    \hline
    $\mathbb{E}[\hat{U}] / \mathbb{E}[U]$ & 91\% &94\%	&94\%	&95\%	&94\%	&94\%
\end{tabular}
\end{minipage}
\end{table}%

\begin{table}[H]
\begin{minipage}{.5\linewidth}
\centering\footnotesize\setlength{\tabcolsep}{3pt}
    \begin{tabular}{c|c|c|c|c|c|c}
    $\lambda$ & 10 & 11 & 12 & 13 & 14 & 15 \\
    \hline
    Freq & 90\%	&89\%	&91\%	&92\%&	90\%	&91\%
\end{tabular}
\end{minipage}
\begin{minipage}{.5\linewidth}
\centering\footnotesize\setlength{\tabcolsep}{3pt}
    \begin{tabular}{c|c|c|c|c|c|c}
    $\lambda$ & 10 & 11 & 12 & 13 & 14 & 15 \\
    \hline
    $\mathbb{P}(U = \hat{U})$ & 49\%	&54\%	&57\%	&61\%&	58\%	&58\%
\end{tabular}
\end{minipage}
\caption{Estimation of $\mathbb{P}(\hat{U} \geq 80\% U)$ and $\mathbb{P}(\hat{U} = U)$}
\end{table}

From a practical point of view, data show that the outcome of our strategy rarely differs from the optimal outcome as shown by the graphs. In at least $89\%$ of auctions, the obtained utility is higher than $80\%$ of the best-possible utility. Furthermore, in at least half of the sample, the strategy achieved the best possible utility.

From those figures, we can conjecture that against an SB-opponent, the expected utility of a Bellman-player is higher than the expected utility of an SB-player (that is $G_{t, \beta^v, \sigma^\nu}(t, h_t) > G_{t, \sigma^v, \sigma^\nu}(t, h_t)$). Therefore, given that the opponent plays SB, the player has no interest to deviate from the Bellman strategy to SB. This would entail that the SB strategy is not an equilibrium. However, the question remains whether the Bellman strategy is at equilibrium.

\section{Conclusion}%

We have modeled a real-life auction against a straightforward but realistic strategy as a POMDP, obtained the optimal strategy, and compared the performances of this strategy to a perfectly informed player. This strategy is robust in the sense that its expected gain does not decrease if opponents change their strategy for a strategy within a reasonable class. Our results on this strategy can pave the way to study an equilibrium of Clock Plus auctions with a single category.

\revision{One may note some limitation to the model as our strategy requires to have some knowledge of the distribution of the opponents' valuation and its computation is faster when this distribution follows a Markov Chain.
Those points can be tough in practice. Firstly, valuations are sensitive information for telecommunication companies, thus knowledge on their distributions is hard to acquire. Secondly, the Markov Chain assumption is strong. As a matter of fact, the Markov Chain assumption supposes that the aggregated demand of the operators does not depend on all the history of the auction. This property has been proven for two distribution models. However, it would not hold for any distribution.
Moreover, valuations may revised during the auction which is a case we did not explore. Indeed, in practice, a company could adjust its budget to match its competitors' if it did not correctly assess and anticipate their bids. Future study could thus try to model more complex valuations such as history-dependent valuations.}

Future work would also focus on the question of equilibrium. Indeed, we would have to compute the expected gain of an SB-player against a Bellman-opponent and compare it to the expected gain using the Bellman strategy. This comparison would either show that the auction has no Nash equilibrium or that the Bellman strategy is a Nash equilibrium. Furthermore, we have compared numerically the Bellman strategy and the perfectly informed play. We could extend this study by having theoretical guarantees such as explicit lower bounds. Lastly, one could study the scaling of this result in higher dimension auctions such as the SAA (Simultaneous Ascending Auction) where players are allowed to bid on multiple items rather than a number of items, raising their prices individually.
\newpage


%
%
%
\section{Appendix}\label{Appendix}

\subsection{Proof of \Cref{lemma1}}

\begin{proof}

Let $v$ a valuation. Let $\forall t, \forall k, \forall \delta, \tilde{\beta}^v(t, k, \delta) = \min(\beta^v(t, k, \delta), \sigma^v(t))$. Let $k_0$ be the initial demand of a $\beta^v$-player.

We define $\begin{cases}
    \beta_0 &= k_0\\
    \beta_{t+1} &= \beta^v(t+1, \beta_t, \sigma^\nu(t))
\end{cases}$
and similarly, $\forall t \geq 0, \sigma_t = \sigma^v(t)$. $(\beta_t)$ represents the sequence of moves for a $\beta^v$-player.
$(\sigma_t)$ represents the sequence of moves for a $\sigma^v$-player. We set $\tilde{\beta}_t = \min(\sigma_{t}, \beta_t)$ which is the sequence of moves for a $\tilde{\beta}^v$-player. Let $T^\beta = \inf\{r \geq 0 | \beta_r + \sigma^\nu(t) \leq m\}$ (resp. $(T^{\tilde{\beta}}=\inf\{r \geq 0 | \tilde{\beta}_r + \sigma^\nu(t) \leq m\}$)) be the stopping time such as the auction terminates at $T^\beta$ (resp. $T^{\tilde{\beta}}$). We immediately have that $T^{\tilde{\beta}} \leq T^\beta$ since $\forall t, \beta_t + \delta_t \geq \tilde{\beta}_t + \delta_t$. Let $U := v({\beta}_{T^{{\beta}}}) - {\beta}_{T^{{\beta}}} p_{T^{{\beta}}-1}$ the utility by playing $\beta$ and $\tilde{U} := v(\tilde{\beta}_{T^{\tilde{\beta}}}) - \tilde{\beta}_{T^{\tilde{\beta}}} p_{T^{\tilde{\beta}}-1}$ the utility by playing $\tilde{\beta}$. 

We have two cases:
\begin{enumerate}
    \item $T^{\tilde{\beta}} = T^{\beta}$. In this case, either $\tilde{\beta}_{T^{\tilde{\beta}}} = \beta_{T^\beta}$ and 
    \begin{equation}
        v(\tilde{\beta}_{T^{\tilde{\beta}}}) - \tilde{\beta}_{T^{\tilde{\beta}}} p_{T^{\tilde{\beta}}-1} = v(\beta_{T^\beta}) - \beta_{T^\beta}p_{T^\beta -1}
    \end{equation}
    or $\tilde{\beta}_{T^{\tilde{\beta}}} = \sigma_{T^{\tilde{\beta}}} =  \sigma_{T^\beta}$, yielding 
    \begin{equation}
    \begin{split}
        v(\tilde{\beta}_{T^{\tilde{\beta}}}) - \tilde{\beta}_{T^{\tilde{\beta}}}p_{T^{\tilde{\beta}}-1} &= v(\sigma_{T^{\tilde{\beta}}}) - \sigma_{T^{\tilde{\beta}}}p_{T^{\tilde{\beta}}-1} \\
        &= v(\sigma_{T^{{\beta}}}) - \sigma_{T^{{\beta}}}p_{T^{{\beta}}-1} \\
        & \geq v(\beta_{T^\beta}) - \beta_{T^\beta}p_{T^\beta -1}
    \end{split}
    \end{equation}
    Hence, $\tilde{U} \geq U$.
    \item $T^{\tilde{\beta}} \neq T^{\beta}$ hence $T^{\tilde{\beta}} < T^{\beta}$. We necessarily have $\tilde{\beta}_{T^{\tilde{\beta}}}= \sigma_{T^{\tilde{\beta}}}$ (otherwise $T^{\tilde{\beta}}$ = $T^{\beta}$) and $p_{T^{\tilde{\beta}}} < p_{T^{{\beta}}}$. therefore,
    \begin{equation}
        \begin{split}
            v(\beta_{T_\beta}) - \beta_{T_\beta}p_{T^\beta-1} &< v(\beta_{T_\beta}) - \beta_{T_\beta}p_{T^{\tilde{\beta}}-1}\\
            &\leq v(\sigma_{T^{\tilde{\beta}}}) - \sigma_{T^{\tilde{\beta}}}p_{T^{\tilde{\beta}}-1}\\
            &= v(\tilde{\beta}_{T^{\tilde{\beta}}}) - \tilde{\beta}_{T^{\tilde{\beta}}}p_{T^{\tilde{\beta}}-1}
        \end{split}
    \end{equation}
\end{enumerate}

therefore, on $\{ T^{\tilde{\beta}} \neq T^{\beta} \}$, $\tilde{U} > U$ and on $\{ T^{\tilde{\beta}} = T^{\beta} \}$, we have $\tilde{U} = U$. Hence, if $\mathbb{P}^t(T^{\tilde{\beta}} \neq T^{\beta}) > 0$ then, $\mathbb{E}[U] < \mathbb{E}[\tilde{U}]$ which is absurd since $\mathbb{E}[U] \geq \mathbb{E}[\tilde{U}]$. Thus, $T^{\tilde{\beta}} = T^{\beta}$ a.s. and $\tilde{U} \geq U$. This implies $\tilde{U} = U$.
Since for all round $t$, for all bids $k$ and $\delta$, $\tilde{\beta}(t,k,\delta) \leq \beta(t,k,\delta)$, we have $\forall t, \forall k, \forall \delta, \tilde{\beta}(t,k,\delta) = \beta(t,k,\delta)$. Furthermore, $\beta = \tilde{\beta}$ verifies $\forall t, \forall k, \forall \delta, \tilde{\beta}(t, k, \delta) \leq \sigma(t)$ which ends the proof.
\end{proof}

\subsection{Proof of \Cref{uniform_lemma}}
\begin{proof}
There exists $\tau$ such that $\forall j \in \{1, ..., m\}, Z_j = \tilde{Z}_{\tau(j)}$ with $\tau$ a random variable of $\mathcal{S}_m$ the set of permutation of $\{1, ..., m\}$.

Let $U$ be an open set of $\mathbb{R}^m$.
\begin{align*}
    \mathbb{P}((Z_1, ..., Z_m) \in U) &= \sum_{\sigma \in \mathcal{S}_m} \mathbb{P}((\tilde{Z}_{\sigma(1)}, ..., \tilde{Z}_{\sigma(m)}) \in U, \sigma = \tau)\\
    &= \sum_{\sigma \in \mathcal{S}_m} \mathbb{P}((\tilde{Z}_{\sigma(1)}, ..., \tilde{Z}_{\sigma(m)}) \in U, \tilde{Z}_{\sigma(1)} \geq ... \geq\tilde{Z}_{\sigma(m)})
\end{align*}

Hence for $\sigma \in \mathcal{S}_m, (\Tilde{Z}_{\sigma(1)}, ...,\Tilde{Z}_{\sigma(m)})$ has the same law as $(\Tilde{Z}_{1}, ...,\Tilde{Z}_m)$. therefore,
\begin{align*}
    \mathbb{P}((Z_1, ..., Z_m) \in U) &= \sum_{\sigma \in \mathcal{S}_m} \mathbb{P}((\tilde{Z}_{1}, ..., \tilde{Z}_{m}) \in U, \tilde{Z}_{1} \geq ... \geq \tilde{Z}_{m})\\
    &= m!\mathbb{P}((\tilde{Z}_{1}, ..., \tilde{Z}_{m}) \in U, \tilde{Z}_{1} \geq ... \geq \tilde{Z}_{m}) \\
    &= m! \int_{\mathbb{R}} f_{\Tilde{Z}_1}(z_1)... f_{\Tilde{Z}_m}(z_m)\mathbf{1}_{z_1 \geq ... \geq z_m}(z)\mathbf{1}_{z \in U}(z) dz\\
    &= \int_U \frac{m!}{Z_{\max}^m}\mathbf{1}_{z \in [0, Z_{\max}]^m}(z) \mathbf{1}_{z_1 \geq ... \geq z_m}(z) dz
\end{align*}
\end{proof}

\subsection{Proof of \Cref{binomial}}
We have
\begin{align*}
  \mathbb{P}(\delta(p_{t+1}) = \delta_{t+1} | \delta(p_t) = \delta_t) &= \frac{((n-1)m)!}{Z_{\max}^{(n-1)m}}\int_{A} \frac{\mathbf{1}_{z_1 \geq ... \geq z_{(n-1)m}}(\mathbf{z}) d\mathbf{z}}{\mathbb{P}(\delta(p_t) = \delta_t)}\end{align*}
where $A$ is defined in the first part of \Cref{uniform}'s proof.

On one hand, by computing the characteristic function of $\delta(p_t)$ $\Phi_{\delta(p_t)}$, we obtain that $\delta(p_t)$ follows a binomial law of parameters $(n-1)m$ and $u = 1 -\frac{p_t}{\Zmax}$.
First notice that for all $i \in \{1, ..., n-1\}$, for all $j \in \{1, ..., m\}$,
\begin{align*}
    \mathbb{E}[e^{it\mathbf{1}(\tilde{Z}^i_j > p_t)}]&=\int_{\mathbb{R}} \exp\Big(it\mathbf{1}(z > p_t)\Big) d\mathbb{P}(z)\\
    &= \int_{0}^{p_t} \frac{1}{Z_{\max}} dz+
    \int_{p_t}^{Z_{\max}} \frac{e^{it}}{Z_{\max}} dz \\
    &= \frac{p_t}{Z_{\max}} + e^{it}\frac{Z_{\max} - p_t}{Z_{\max}} = \frac{p_t}{Z_{\max}} + e^{it}\big(1 - \frac{p_t}{Z_{\max}}\big)
\end{align*}
Then,
\begin{align*}
    \forall t \in \mathbb{R}, \Phi_{\delta(p_t)}(t) &= \mathbb{E}[e^{it\delta(p_t)}]\\
    &= \mathbb{E}[e^{it\sum_{j = 1}^{(n-1)m} \mathbf{1}(Z_j > p_t)}]\\
    &= \mathbb{E}[\prod_{j = 1}^{(n-1)m}e^{it\mathbf{1}(Z_j > p_t)}]\\
    &= \mathbb{E}[\prod_{i = 1}^{n-1}\prod_{j = 1}^{m}e^{it\mathbf{1}(\tilde{Z}^i_j > p_t)}]\\
    &= \prod_{i = 1}^{n-1}\prod_{j = 1}^{m}\mathbb{E}[e^{it\mathbf{1}(\tilde{Z}^i_j > p_t)}]\\
    \Phi_{\delta(p_t)}(t) &= \Big(e^{it}\big(1 - \frac{p_t}{Z_{\max}}\big) + \frac{p_t}{Z_{\max}}\Big)^{(n-1)m}
\end{align*}
Hence, $\mathbb{P}(\delta(p_t) = \delta_t) = \binom{(n-1)m}{\delta_t} \big(1 - \frac{p_t}{\Zmax}\big)^{\delta_t}\big(\frac{p_t}{\Zmax}\big)^{(n-1)m - \delta_t}$.

On the other hand,
\begin{eqnarray*}
\lefteqn{\int_{A} \mathbf{1}_{z_1 \geq ... \geq z_m}(\mathbf{z}) d\mathbf{z} = \int_{A} \mathbf{1}_{z_1 \geq ... \geq z_{\delta_{t+1}}}(\mathbf{z})\mathbf{1}_{z_{\delta_{t+1} + 1} \geq ... \geq z_{\delta_{t}}} (\mathbf{z})\mathbf{1}_{z_{\delta_{t} + 1} \geq ... \geq z_{m}} (\mathbf{z}) d\mathbf{z}}\\
    &=& \Big(\int_{(p_{t+1}, Z_{\max}]^{\delta_{t+1}}} \mathbf{1}_{z_1 \geq ... \geq z_{\delta_{t+1}}}(z_1, ..., z_{\delta_{t+1}}) d(z_1, ..., z_{\delta_{t+1}})\Big)\\
    &&\times \Big(\int_{(p_t, p_{t+1}]^{\delta_t - \delta_{t+1}}} \mathbf{1}_{z_{\delta_{t+1} + 1} \geq ... \geq z_{\delta_{t}}}(z_{\delta_{t+1} + 1}, ..., z_{\delta_{t}}) d(z_{\delta_{t+1} + 1}, ..., z_{\delta_{t}})\Big)\\
    &&\times  \Big(\int_{(0, p_t]^{(n-1)m - \delta_t}} \mathbf{1}_{z_{\delta_{t} + 1} \geq ... \geq z_{m}}(z_{\delta_{t} + 1}, ..., z_{m}) d(z_{\delta_{t} + 1}, ..., z_{m})\Big)
\end{eqnarray*}
Let us define for all $k$, $u_k = \frac{z_k - p_{t+1}}{Z_{\max} - p_{t+1}}$ and $\alpha(u_k) = p_{t+1} + (\Zmax - p_{t+1})u_k$.
\begin{align*}
    I_1 &= \int_{(p_{t+1}, Z_{\max}]^{\delta_{t+1}}} \mathbf{1}_{z_1 \geq ... \geq z_{\delta_{t+1}}}(z_1, ..., z_{\delta_{t+1}}) d(z_1, ..., z_{\delta_{t+1}})\\
    &= (Z_{\max} - p_{t+1})^{\delta_{t+1}} \int_{(0, 1]^{\delta_{t+1}}} \mathbf{1}_{\alpha(u_1) \geq ... \geq \alpha(u_{\delta_{t+1}})}(u_1, ..., u_{\delta_{t+1}}) d(u_1, ..., u_{\delta_{t+1}})\\
    &= (Z_{\max} -p_{t+1})^{\delta_{t+1}} \int_{(0, 1]^{\delta_{t+1}}} \textbf{1}_{u_1 \geq ... \geq u_{\delta_{t+1}}}(u_1, ..., u_{\delta_{t+1}}) d(u_1, ..., u_{\delta_{t+1}})\\
    &= (Z_{\max} - p_{t+1})^{\delta_{t+1}} \int_{(0, 1]} \int_{(0, u_1]} \Big( ... \Big( \int_{(0, u_{\delta_{t+1} - 1}]} 1 du_{\delta_{t+1}}\Big)du_{\delta_{t+1} - 1}\Big) ... \Big) du_1\\
    &= (Z_{\max} - p_{t+1})^{\delta_{t+1}} \big[ \frac{u_{1}^{\delta_{t+1}}}{\delta_{t+1}!} ]^1_0\\
    &= \frac{(Z_{\max} - p_{t+1})^{\delta_{t+1}}}{(\delta_{t+1})!}
\end{align*}
We can derive $I_2$ and $I_3$ using the same method: $I_2 = \frac{(p_{t+1} - p_t)^{\delta_t - \delta_{t+1}}}{(\delta_t - \delta_{t+1})!}$ and $I_3 = \frac{p_t^{(n-1)m - \delta_t}}{((n-1)m - \delta_t)!}$. Lastly,

$$ \int_{A} \mathbf{1}_{z_1 \geq ... \geq z_m}(\mathbf{z}) d\mathbf{z} =  \frac{(Z_{\max} - p_{t+1})^{\delta_{t+1}}(p_{t+1} - p_t)^{\delta_t - \delta_{t+1}}p_t^{(n-1)m - \delta_t}}{(\delta_{t+1})!(\delta_t - \delta_{t+1})!((n-1)m - \delta_t)!}$$

Thus, letting $\mathbb{P}_{\delta_{t+1}|\delta_t}^t = \mathbb{P}(\delta(p_{t+1}) = \delta_{t+1} | \delta(p_{t}) = \delta_{t} )$, we have:
\begin{align*}
    \mathbb{P}_{\delta_{t+1}|\delta_t}^t &= \frac{((n-1)m)!}{Z_{\max}^{(n-1)m}} \frac{\int_{A} \mathbf{1}_{z_1 > ... > z_m}(\mathbf{z}) d\mathbf{z}}{\binom{(n-1)m}{\delta_t} \big(1 - \frac{p_t}{Z_{\max}}\big)^{\delta_t} \big(\frac{p_t}{Z_{\max}}\big)^{(n-1)m -\delta_t} }\\
    &= \frac{((n-1)m)!}{Z_{\max}^{(n-1)m}}\frac{\frac{(Z_{\max} - p_{t+1})^{\delta_{t+1}}(p_{t+1} - p_t)^{\delta_t - \delta_{t+1}}p_t^{(n-1)m - \delta_t}}{(\delta_{t+1})!(\delta_t - \delta_{t+1})!((n-1)m - \delta_t)!}}{\binom{(n-1)m}{\delta_t}(1 - \frac{p_t}{Z_{\max}})^{\delta_t}(\frac{p_t}{Z_{\max}})^{(n-1)m-\delta_t}}\\
   \mathbb{P}_{\delta_{t+1}|\delta_t}^t &= \binom{\delta_t}{\delta_{t+1}} \frac{(\Zmax - p_{t+1})^{\delta_{t+1}}(p_{t+1} - p_t)^{\delta_t - \delta_{t+1}}}{(\Zmax - p_t)^{\delta_t}}\\
    &= \binom{\delta_t}{\delta_{t+1}} \big( \frac{\Zmax - p_{t+1}}{\Zmax - p_{t}}\big)^{\delta_{t+1}}\big( 1 - \frac{\Zmax - p_{t+1}}{\Zmax - p_{t}}\big)^{\delta_t - \delta_{t+1}}
\end{align*}

\subsection{Proof of \Cref{proba_exp}}

First, $\mathbb{P}(\delta(p_{t+1}) = \delta | \delta(p_t) = 0) = 1 \iff \delta = 0$ since $(\delta(p))$ is non-increasing and non-negative. Let us fix $\delta_t > 0$ and $\delta_{t+1} \in \{1, ..., \delta_t\}$. We denote by $\mathbb{P}_{\delta_{t+1}|\delta_t}^t$ the probability $\mathbb{P}(\delta(p_{t+1}) = \delta_{t+1} | \delta(p_t) = \delta_t)$.

\begin{align*}
     \mathbb{P}_{\delta_{t+1}|\delta_t}^t &= \mathbb{P}((n-1)m - N_{p_{t+1} \wedge Z_1} = \delta_{t+1} | (n-1)m - N_{p_{t} \wedge Z_1} = \delta_{t})\\
    &= \mathbb{P}( N_{p_{t+1} \wedge Z_1} = (n-1)m -\delta_{t+1} |  N_{p_{t} \wedge Z_1} = (n-1)m -\delta_{t})\\
    &= \mathbb{P}( N_{p_{t+1}} = (n-1)m -\delta_{t+1} |  N_{p_{t}} = (n-1)m -\delta_{t})\\
    &= \frac{\mathbb{P}( N_{p_{t+1}} = (n-1)m -\delta_{t+1}, N_{p_{t}} = (n-1)m -\delta_{t})}{\mathbb{P}(N_{p_{t}} = (n-1)m -\delta_{t})}\\
    &= \frac{\mathbb{P}( N_{p_{t+1}} - N_{p_{t}} = \delta_{t} - \delta_{t+1}, N_{p_{t}} = (n-1)m -\delta_{t})}{\mathbb{P}(N_{p_{t}} = (n-1)m -\delta_{t})}\\
    &= \frac{\mathbb{P}( N_{p_{t+1}} - N_{p_{t}} = \delta_{t} - \delta_{t+1})\mathbb{P}( N_{p_{t}} = (n-1)m -\delta_{t})}{\mathbb{P}(N_{p_{t}} = (n-1)m -\delta_{t})}\\
    &= \mathbb{P}( N_{p_{t+1}} - N_{p_{t}} = \delta_{t} - \delta_{t+1})\\
    &= e^{-\lambda(p_{t+1} - p_t)}\frac{(\lambda(p_{t+1} - p_t))^{\delta_{t} - \delta_{t+1}}}{(\delta_{t} - \delta_{t+1})!}\\
    &= e^{-\lambda\Delta P}\frac{(\lambda\Delta P)^{\delta_{t} - \delta_{t+1}}}{(\delta_{t} - \delta_{t+1})!}
\end{align*}

Now,

\begin{align*}
    \mathbb{P}_{0|\delta_t}^t &= \mathbb{P}((n-1)m - N_{p_{t+1} \wedge Z_1} = 0 | (n-1)m - N_{p_{t} \wedge Z_1} = \delta_t)\\
    &= \mathbb{P}(N_{p_{t+1} \wedge Z_1} = (n-1)m |N_{p_{t} \wedge Z_1} = (n-1)m - \delta_t)\\
    &= \mathbb{P}(N_{p_{t+1}} \geq (n-1)m | N_{p_t} = (n-1)m - \delta_t)\\
    &= \frac{\mathbb{P}(N_{p_{t+1}} \geq (n-1)m, N_{p_t} = (n-1)m - \delta_t)}{\mathbb{P}(N_{p_t} = (n-1)m - \delta_t)}\\
    &= \frac{\mathbb{P}(N_{p_{t+1}} - N_{p_t} \geq \delta_t, N_{p_t} = (n-1)m - \delta_t)}{\mathbb{P}(N_{p_t} = (n-1)m - \delta_t)}\\
    &= \frac{\mathbb{P}(N_{p_{t+1}} - N_{p_t} \geq \delta_t)\mathbb{P}(N_{p_t} = (n-1)m - \delta_t)}{\mathbb{P}(N_{p_t} = (n-1)m - \delta_t)}\\
    &= \mathbb{P}(N_{p_{t+1}} - N_{p_t} \geq \delta_t)\\
    &= \sum_{j = \delta_t}^{+\infty} \mathbb{P}(N_{p_{t+1}} - N_{p_t} = j)\\
    &= \sum_{j = \delta_t}^{+\infty} e^{-\lambda \Delta P}\frac{(\lambda \Delta P)^j}{j!}
\end{align*}

All in all, \begin{align*}
    \mathbb{P}(\delta(p_{t+1}) = \delta_{t+1} | \delta(p_t) = \delta_t ) &= \begin{cases}
        e^{-\lambda \Delta P}\sum_{j = \delta_t}^{+\infty} \frac{(\lambda \Delta P)^j}{j!} &\text{ if } \delta_{t+1} = 0\\
        e^{-\lambda \Delta P}\frac{(\lambda \Delta P) ^{\delta_t - \delta_{t+1}}}{(\delta_t - \delta_{t+1})!} &\text{ otherwise}
    \end{cases}
\end{align*}

\end{document}